\documentclass[a4paper, twoside, 11pt]{article}
\usepackage{titlesec, longtable}
\usepackage[hidelinks]{hyperref}

\titleclass{\subsubsubsection}{straight}[\subsection]
\newcounter{subsubsubsection}[subsubsection]

\renewcommand\thesubsubsubsection{\thesubsubsection.\arabic{subsubsubsection}}

\titleformat{\subsubsubsection}
 {\normalfont\normalsize\bfseries}{\thesubsubsubsection}{1em}{}
\titlespacing*{\subsubsubsection}
{0pt}{3.25ex plus 1ex minus .2ex}{1.5ex plus .2ex}

\makeatletter
\renewcommand\paragraph{\@startsection{paragraph}{5}{\z@}%
 {3.25ex \@plus1ex \@minus.2ex}%
 {-1em}%
 {\normalfont\normalsize\bfseries}}
\renewcommand\subparagraph{\@startsection{subparagraph}{6}{\z@}
 {3.25ex \@plus1ex \@minus .2ex}%
 {-1em}%
 {\normalfont\normalsize\bfseries}}
\def\toclevel@subsubsubsection{4}
\def\toclevel@paragraph{5}
\def\toclevel@paragraph{6}
\def\l@subsubsubsection{\@dottedtocline{4}{7em}{4em}}
\def\l@paragraph{\@dottedtocline{5}{10em}{5em}}
\def\l@subparagraph{\@dottedtocline{6}{14em}{6em}}
\@addtoreset{subsubsubsection}{section}
\@addtoreset{subsubsubsection}{subsection}
\@addtoreset{paragraph}{subsubsubsection}
\makeatother

\usepackage{a4wide, fancyhdr, amsmath, amssymb, mathtools, yfonts, bbm}
\usepackage{mathrsfs}
\usepackage{graphicx}
\usepackage{tikz}
\usepackage[all]{xy}
\usepackage[utf8]{inputenc}
\usepackage{amsthm}
\usepackage[english]{babel}
\usepackage{chngcntr}
\usepackage{ifthen}
\usepackage{calc}
\usepackage{authblk}
\newcommand{\C}{\mathbb{C}}

\newcommand{\Z}{\mathbb{Z}}
\newcommand{\Q}{\mathbb{Q}}

\newcommand\GL{\mathrm{GL}_2(\mathbb{Z})}

\numberwithin{equation}{section}
\newtheorem{lemma}{Lemma}[section]
\newtheorem{theorem}[lemma]{Theorem}
\newtheorem{prop}[lemma]{Proposition}
\newtheorem{corollary}[lemma]{Corollary}

\newtheorem{remark}[lemma]{Remark}
\newtheorem{example}[lemma]{Example}

\newtheorem{definition}[lemma]{Definition}
\newtheorem{thmx}{Theorem}

\title{\vspace{-\baselineskip}\sffamily\bfseries Binary forms with the same value set I}
\author[1]{\'Etienne Fouvry\thanks{CNRS, Laboratoire de math\' ematiques d'Orsay, Universit\' e Paris--Saclay, 91405 Orsay, France, etienne.fouvry@universite-paris-saclay.fr}}
\author[2]{Peter Koymans\thanks{Institute for Theoretical Studies, ETH Zurich, 8092 Zurich, Switzerland, peter.koymans@eth-its.ethz.ch}}
\affil[1]{Universit\'e Paris--Saclay}
\affil[2]{ETH Zurich}
\date{\today}
\begin{document}
\maketitle

\begin{abstract}
Given a binary form $F \in \Z[X, Y]$, we define its value set to be $\{F(x, y) : (x, y) \in \Z^2\}$. Let $F, G \in \Z[X, Y]$ be two binary forms of degree $d \geq 3$ and with non-zero discriminant. In a series of three papers, we will give necessary and sufficient conditions on $F$ and $G$ to have the same value set. These conditions will be entirely in terms of the automorphism groups of the forms. 

In this paper, we will build the general theory that reduces the problem to a question about lattice coverings of $\Z^2$, and we solve this problem when $F$ and $G$ have a small automorphism group. The larger automorphism groups $D_4$ and $D_3, D_6$ will respectively be treated in part II and part III.
\end{abstract}

\section{Introduction}
\subsection{Origin of the problem} 
The idea of this work emerged when studying the work of Watson \cite{Wa1} where it is shown that if two positive definite binary quadratic forms $F$ and $G$ with real (not necessarily integral) coefficients have the same value set, which means $F(\Z^2) = G (\Z^2)$, and are not ${\rm GL}(2, \Z)$--equivalent, then they respectively are ${\rm GL}(2, \Z)$--equivalent to
\begin{equation}
\label{Del-Wat}
c(X^2+XY+Y^2) \text { and } c(X^2+3Y^2),
\end{equation}
for some real number $c > 0$. Actually, Watson \cite{Wa2} realized that this result was already known to Delone (see for instance \cite[Theorem p.163]{De}). This is the reason why we would like to give the name {\it Delone--Watson forms} to the forms appearing in \eqref{Del-Wat}.

The case of indefinite binary quadratic forms, with integral coefficients and with equal images apparently seems to have not been discussed in the literature set apart from the work of Schering giving general criteria for two forms, definite or indefinite, to have the same value set (\cite{Sch}, also mentioned in \cite[p. 26]{Di}). The case of indefinite quadratic forms exhibits completely different situations as it will be shown in \cite{FoKo} by appealing to recent statistics on the class groups of real quadratic fields.

In this paper we are concerned with binary forms with degree at least $3$. For a precise statement of our results, we introduce the following notations and definitions.

If $d \geq 1$ is an integer, we denote by ${\rm Bin}(d, \Q)$ the set of binary forms $F = F(X, Y)$ with degree $d$, with rational coefficients and with discriminant different from zero. The group ${\rm GL}(2, \Q)$ operates on ${\rm Bin}(d, \Q)$ by the natural change of variables: if $\gamma = \begin{pmatrix} a & b \\ c& d \end{pmatrix}$ belongs to ${\rm GL}(2, \Q)$ then, by definition
$$
(F \circ \gamma)(X, Y) := F(aX + bY, cX + dY).
$$
The {\it value set} (or {\it image}) of the form $F$ is by definition $F(\Z^2) := \{F(x, y) : (x, y) \in \Z^2\}$ and it satisfies the equality
\begin{equation}
\label{actionofgamma}
(F \circ \gamma)(\Z^2) = F(\Z^2)
\end{equation}
for every $\gamma \in {\rm GL}(2, \Z)$. In this paper, we are concerned with the following question, which in some sense can be interpreted as the converse of \eqref{actionofgamma} or an extension of the Delone--Watson result to higher degrees:
\vskip .2cm
\textit{Let $d \geq 3$ and let $F$ and $G$ be two forms of ${\rm Bin}(d, \Q)$ such that $F(\Z^2) = G(\Z^2)$. Does there exist $\gamma \in {\rm GL}(2, \Z)$
such that}
\begin{equation} 
\label{G= Fogamma}
G = F \circ \gamma\, ?
\end{equation}
We can always suppose that $F$ and $G$ both belong to ${\rm Bin}(d, \Z)$, by multiplication by some integer different from zero, to eliminate the denominators.

\begin{definition}
\label{defextraordinary} 
Let $d \geq 3$. A form $F \in {\rm Bin}(d, \Q)$ is said to be {\it extraordinary} if there exists $G \in {\rm Bin}(d, \Q)$ satisfying $G(\Z^2) = F(\Z^2)$ and no $\gamma \in {\rm GL}(2, \Z) $ satisfying $G = F \circ \gamma$. A form of ${\rm Bin}(d, \Z)$, which is not extraordinary, is called {\it ordinary}.
\end{definition} 

In Definition \ref{defextraordinary}, the word {\it extraordinary} has been chosen to indicate the scarcity of these forms, as it will appear below. The purpose of this paper is to classify extraordinary forms entirely in terms of elementary invariants attached to the form. It is more convenient to express the results in terms of the following two equivalence classes. For $F, G \in {\rm Bin}(d, \Q)$, we say that

\begin{itemize}
\item $F$ and $G$ have the {\it same value set} or {\it the same image} if and only if one has $F(\Z^2) = G(\Z^2)$. This equivalence relation is denoted by $F \sim_{{\rm val}} G$ and the equivalence class of $F$ in ${\rm Bin}(d, \Q)$ is denoted by $[F]_{\rm val}$,
\item $F$ and $G$ are ${\rm GL}(2, \Q)$--{\it equivalent} (resp. ${\rm GL}(2, \Z)$--{\it equivalent}) if and only if there exists $\gamma$ in ${\rm GL}(2, \Q)$ (resp. $ {\rm GL}(2, \Z)$) such that $G= F \circ \gamma$. This equivalence relation is denoted by $ \sim_{{\rm GL}(2, \Q)} $ (resp. $ \sim_{{\rm GL}(2, \Z)} $) and the corresponding equivalence class is denoted by $[F]_{{\rm GL}(2, \Q)}$ (resp. $[F]_{{\rm GL}(2, \Z)}$).
\end{itemize}

\noindent With these notations \eqref{actionofgamma} is equivalent to the inclusion
$$
[F ]_{{\rm GL}(2, \Z)} \subseteq [F]_{\rm val},
$$
and $F$ is extraordinary if and only if one has the strict inclusion
$$
[F ]_{{\rm GL}(2, \Z)} \subsetneq [F]_{\rm val}.
$$
To state our results when $d \geq 3$, we introduce some important algebraic notions attached to a binary form $F$: an element $\gamma \in {\rm GL}(2, \Q)$ is a $\Q$--{\it automorphism of the form $F$ }if it satisfies the equality $F \circ \gamma = F$. The set of all $\Q$--automorphisms of $F$ forms the {\it group of automorphisms}, denoted ${\rm Aut}(F, \Q)$. We also have the conjugation formula
\begin{equation}
\label{conjugationformula}
{\rm Aut}(F \circ \lambda, \Q) = \lambda^{-1} {\rm Aut}(F, \Q) \, \lambda, \text{ for all } \lambda \in {\rm GL}(2, \Q).
\end{equation}
Similarly, we speak of {\it $\Z$--automorphism of the form $F$} and of the group ${\rm Aut}(F, \Z)$, which is a subgroup of ${\rm Aut}(F, \Q)$. The formula \eqref{conjugationformula} remains true when replacing $\Q$ by $\Z$.

If $G$ and $H$ are two subgroups of ${\rm GL}(2, \Q)$, the notation
$$
G \simeq_{{\rm GL}(2, \Q)} H\ (\text {resp. } G \simeq_{{\rm GL}(2, \Z)} H)
$$
means that there exists an element $k$ of ${\rm GL}(2, \Q)$ (resp. ${\rm GL}(2, \Z)$) such that $G = kH k^{-1}$. If it is the case, we say that $G$ and $H$ are {\it ${\rm GL}(2, \Q)$--conjugate }(resp. {\it ${\rm GL}(2, \Z)$--conjugate}). Similarly, we speak of elements of ${\rm GL}(2, \Q)$ which are {\it ${\rm GL}(2, \Q)$--conjugate} (resp. {\it ${\rm GL}(2, \Z)$--conjugate}) and we use the same notation as for subgroups.

In Lemma \ref{possibleAut} below, we explicitly give the complete list $\mathfrak H$ of ten subgroups of ${\rm GL}(2, \Z)$ with the property that for every $F \in {\rm Bin}(d, \Q)$, the group ${\rm Aut}(F, \Q)$ is ${\rm GL}(2, \Q)$--conjugate to a unique element of $\mathfrak H$. The presentation of our results and of our proofs mandate the following partition of $\mathfrak H$ into three subsets 
\begin{equation}
\label{frakH}
\mathfrak H = \mathfrak H_1 \sqcup \mathfrak H_2 \sqcup \mathfrak H_3,
\end{equation}
with
\begin{gather*}
\mathfrak H_1 := \{{\bf C}_1, {\bf C}_2, {\bf C}_3, {\bf C}_4, {\bf C}_6, {\bf D}_1, {\bf D}_2\}, \\
\mathfrak H_2 := \{ {\bf D}_4\}, \\
\mathfrak H_3 := \{{\bf D}_3, {\bf D}_6\}.
\end{gather*}
In the present paper, we are mainly interested with the possible existence of extraordinary forms $F$ when ${\rm Aut}(F, \Q)$ is ${\rm GL}(2, \Q)$--conjugate to one of the seven groups of the list $\mathfrak H_1$. The cases $\mathfrak H_2$ and $\mathfrak H_3$ require more delicate techniques concerning sets of lattices covering $\Z^2$. They will be treated in two forthcoming works \cite{FKD4} and \cite{FKD3D6}, for comments, see \S \ref{othercases} and \S \ref{strategy}.

For $i \in \{1, 2, 3\}$, we introduce the following condition concerning a form $F$ of ${\rm Bin}(d, \Q)$
$$
(Ci) : \text{ there exists } H \in \mathfrak H_i \text{ such that } {\rm Aut}(F, \Q) \simeq_{{\rm GL}(2, \Q)} H.
$$

\subsection{Statement of the results}
Here is the first characterization of extraordinary forms. Actually this corollary remains true if one replaces the condition $(C1)$ concerning ${\rm Aut}(F, \Q)$ by $(C2)$ or $(C3)$. This will be respectively shown in the two forthcoming papers \cite{FKD4} and \cite{FKD3D6}. We have

\begin{corollary}
\label{illustration1} 
Let $d \geq 3$. Suppose that $F \in {\rm Bin}(d, \Q)$ satisfies the condition $(C1)$. Then $F$ is extraordinary if and only if ${\rm Aut}(F, \Q)$ contains an element 
\begin{equation}
\label{writingsigma}
\sigma =\begin{pmatrix}
a &b \\ c& d
\end{pmatrix}
\end{equation}
with the following properties
\begin{enumerate}
\item the order of $\sigma$ is equal to $3$,
\item the quadruple $(a, b, c, d)$ of rational numbers satisfies one of the four properties
\begin{enumerate}
\item \label{1} $(a, b, c, d) \in \Z^4$,
\item \label{2} $(a, d) \in \Z^2$ and $(b, c) \in \Z \times \left(\Z + \frac{1}{2}\right)$,
\item \label{3} $(a, d) \in \Z^2$ and $(b, c) \in \left(\Z + \frac{1}{2}\right) \times \Z$,
\item \label{4} $(a, b, c, d) \in \left(\Z + \frac{1}{2}\right)^4$.
\end{enumerate}
\end{enumerate}
\end{corollary}

Since $\sigma$ is of order $3$ if and only if $a + d = -1$ and $ad - bc =1$, each of the four possibilities \eqref{1}, \eqref{2}, \eqref{3} and \eqref{4} hold infinitely often. The only groups of the list $\mathfrak H_1$ with cardinality divisible by $3$ are ${\bf C}_3$ and ${\bf C}_6$, thus we directly obtain the next consequence.

\begin{corollary} 
Let $d \geq 3$. There is no extraordinary form $F \in {\rm Bin}(d, \Q)$ such that ${\rm Aut}(F, \Q)$ is ${\rm GL}(2, \Q)$--conjugate to ${\bf C}_1$, ${\bf C}_2,$ ${\bf C}_4,$ ${\bf D}_1$ or ${\bf D}_2$.
\end{corollary}

Let $F$ be an extraordinary form. In \S \ref{theifpart}, we give a construction of a form $G$ satisfying $F(\Z^2) = G(\Z^2)$, but $G \not \sim_{{\rm GL}(2, \Z)} F$. The formula depends on the condition $\eqref{1}, \dots, \eqref{4}$ satisfied by the form $F$. 

The next criterion is the companion of Corollary \ref{illustration1}. It gives an explicit decomposition of $[F]_{\rm val}$ in terms of distinct equivalence classes for $\sim_{{\rm GL}(2, \Z)}$. When $F$ is extraordinary, this number of ${{\rm GL}(2, \Z)}$--equivalence classes is always two (a situation to be compared with \eqref{Del-Wat}).

\begin{corollary} 
\label{illustration2} 
Let $d \geq 3$. A form $F \in {\rm Bin}(d, \Z)$ satisfying $(C1)$ is extraordinary if and only if there exists $G \in [F]_{\rm val}$ such that
\begin{equation}
\label{268}
\{\sigma \in {\rm Aut}(G, \Q) : \sigma^3 = {\rm id}\} \subseteq {\rm GL}(2, \Z) \textup{ and } 3 \mid \vert {\rm Aut}(G, \Q)\vert. 
\end{equation}
Furthermore, we have a decomposition into two disjoint classes $\sim_{{\rm GL}(2, \Z)}$:
\begin{equation}
\label{269}
[F]_{\rm val} = [G]_{{\rm GL}(2, \Z)} \cup [H]_{{\rm GL}(2, \Z)},
\end{equation}
where $H$ is the binary form $H(X, Y) = G(2X, Y)$.
\end{corollary}

Here and also in Theorem \ref{source} below, we will replace the condition $(C1)$ by $(C2)$ or $(C3)$ in \cite {FKD4} and \cite{FKD3D6}. Actually, these two corollaries are consequences of the following theorem, the proof of which occupies the most important part of this paper.

\begin{theorem}
\label{source}
Let $d \geq 3$. Let $F_1 \in {\rm Bin}(d, \Q)$ be an extraordinary form satisfying $(C1)$. Let $F_2 \in {\rm Bin}(d, \Q)$ be such that 
\begin{equation}
\label{hypoforF1F2}
F_1 \sim_{\rm val} F_2 \textup{ and } F_1 \not \sim_{{\rm GL}(2, \Z)} F_2.
\end{equation} 
Then we have $3 \mid \vert{\rm Aut}(F_1, \Q)\vert$. Furthermore, there exist two forms $G_1$ and $G_2$ such that $G_i \sim_{{\rm GL}(2, \Z)} F_i$ and such that
\begin{equation}
\label{G1 = G2}
G_1(X, 2Y) = G_2(X, Y) \textup{ or } G_1(X, 2Y) = G_2(2X, 2Y).
\end{equation}
In the first case we have
$$
\{\sigma \in {\rm Aut}(G_1, \Q) : \sigma^3 = {\rm id}\} \subseteq {\rm GL}(2, \Z),
$$
and in the second case, we have
$$
\{\sigma \in {\rm Aut}(G_2, \Q) : \sigma^3 = {\rm id}\} \subseteq {\rm GL}(2, \Z).
$$
\end{theorem}

\begin{remark} 
The second case case of \eqref{G1 = G2} is symmetrical with the first one: the relation $G_1(X, 2Y) = G_2(2X, 2Y)$ is equivalent to $G_1(X, Y) = G_2(2X, Y)$. Let, for $i \in \{1, 2\}$, $H_i$ be the binary form
$$
H_i = G_i \circ \begin{pmatrix} 0 &1 \\ 1 & 0 \end{pmatrix}.
$$ 
The $H_i$ are ${\rm GL}(2, \Z)$--equivalent to the $G_i$ and they satisfy the equality
$$
H_2(X, 2Y) = H_1(X, Y).
$$
We recognize the first case of \eqref{G1 = G2} after inverting the indices.
\end{remark}

It is also natural to wonder what happens for \emph{coprime value sets}. This question was also considered by Watson \cite[p. 73]{Wa1}, who observed that the Delone--Watson forms have different coprime value sets. 

More formally, define
\[
W(F) = \{F(x, y) : x, y \in \Z, \gcd(x, y) = 1\}.
\]
This induces an equivalence relation $\sim_{\text{co}}$ on the set of binary forms by setting $F \sim_{\text{co}} G$ if $W(F) = W(G)$. We have inclusions
\begin{align}
\label{eCoprimeValue}
[F]_{\GL} \subseteq [F]_{\text{co}} \subseteq [F]_{\text{val}}.
\end{align}
The first inclusion uses that $\GL$ permutes the set of primitive elements of $\Z^2$, while the second inclusion uses homogeneity. The following result shows that the situation becomes completely rigid for coprime value sets.

\begin{theorem}
\label{tCoprimeValue}
Let $d \geq 3$ and let $F \in {\rm Bin}(d, \Q)$. Then we have $[F]_{\GL} = [F]_{\textup{co}}$.
\end{theorem}

In words, a binary form is completely determined (up to $\GL$) by its coprime value set. Our proof of Theorem \ref{tCoprimeValue} will rest in an essential way on our results for $\sim_{\text{val}}$. We will prove Theorem \ref{tCoprimeValue} in \S \ref{sCoprime}.
%%%%%%%%%%%%%%%%%

A second variation of the notion of equality of value sets concerns {\it value sets with multiplicities} that was suggested to us by Michel Waldschmidt. For $F \in {\rm Bin}(d, \Z)$ and $m$ an integer different from zero, we introduce the finite number
$$
R(F; m) := \vert \{(x, y) \in \Z^2: F(x, y) = m\} \vert.
$$
We say the forms $F$ and $G\in {\rm Bin}(d, \Z)$ have the {\it same value sets with multiplicities} if for every $m\neq 0$, one has the equality
$$
R(F;m) =R (G; m).
$$
If this condition is satisfied, we write $F \sim_{{\rm mu}} G$ and the analogue of \eqref{eCoprimeValue} is the following straightforward inclusions between equivalence classes
$$
[F]_{\GL} \subseteq [F]_{\text{mu}} \subseteq [F]_{\text{val}}.
$$
Here also the situation becomes completely rigid for value sets with multiplicities.

\begin{theorem}
\label{tMultiValue}
Let $d \geq 3$ and let $F \in {\rm Bin}(d, \Z)$. Then we have $[F]_{\GL} = [F]_{\textup{mu}}$.
\end{theorem}

This means that a binary form is completely determined (up to $\GL$) by its value set with multiplicities. The proof of Theorem \ref{tMultiValue} will rest in an essential way on our results for $\sim_{\text{val}}$. We will prove Theorem \ref{tMultiValue} in \S \ref{sMulti}.
 
\subsection{Some comments and illustrations} 
We will frequently meet the matrix
\begin{equation}
\label{defR}
R := \begin{pmatrix} 0& 1\\-1 & -1
\end{pmatrix}.
\end{equation}
This matrix is of order $3$ and it generates the cyclic subgroup of ${\rm GL}(2, \Z)$ called ${\bf C}_3 := \{{\rm id}, R, R^2\}$, see Lemma \ref{possibleAut}.

\begin{example}
We return to the Delone--Watson forms appearing in \eqref{Del-Wat}, with $c = 1$. We respectively call them $F$ and $G$. By computing their discriminants, these forms are not ${\rm GL}(2, \Z)$--equivalent. We have the equalities
$$
(G \circ R^2) (X, Y) = (-X-Y)^2 + 3X^2 = 4X^2 +2XY +Y^2 =: H(X, Y)
$$
with $H(X, Y) = F(2X, Y)$. So we have $G \sim_{{\rm GL}(2, \Z)} H$ and $F \not \sim_{{\rm GL}(2, \Z)} H$. To prove that $F(\Z^2) = G(\Z^2)$, it is sufficient to prove that $F(\Z^2) = H(\Z^2)$. We obviously have $H(\Z^2) \subseteq F(\Z^2)$. For the opposite inclusion, we use the fact that $F = F \circ R = F \circ R^2 $, which implies that, for any integers $m$ and $n$, we have the equalities
\begin{equation}
\label{G= G= G}
F(m, n) = F(n, -m - n) = F(-m - n, m).
\end{equation}
For any $(m, n) \in \Z^2$, at least one of the integers $m$, $n$ and $-m-n$ is even, thus we can always write $F(m, n) = F(2m', n') = F(m', n')$ for some integers $m'$ and $n'$. So we have the equality $F(\Z^2) = H(\Z^2)$. By \eqref{Del-Wat} we have the decomposition 
$$
[F]_{\rm val} = [G]_{\rm val} = [H]_{\rm val} = [F]_{{\rm GL}(2, \Z)} \cup [H]_{{\rm GL}(2, \Z)}.
$$
Such a decomposition, in degree $d = 2$, shows the analogy between the content of Corollary \ref{illustration2} in degree $d \geq 3$ and the Delone--Watson result.
\end{example}

\begin{example}
\label{example1.7}
Now we pass to the case where $d = 3$. A direct computation shows that a binary form $F$ over $\Q$ with degree $3$ or $-\infty$ is stable by the action of ${\bf C}_3$ if and only if it is of the form
$$
F(X, Y) = F_{a, b}(X, Y) := aX^3 + b X^2Y + (b - 3a) XY^2 - a Y^3 \quad (a, b \in \Q),
$$
see for example \cite[p.818]{S}. Its discriminant is equal to $(9a^2 - 3ab + b^2)^2$. Choose $a = 1$ and $b = 0$ so that\footnote{The polynomial $F_{1, m}(x, 1)$ is the polynomial $f_m(x)$ introduced in \cite[p. 712]{FrIwMaRu} in a different context to parametrize cubic cyclic number fields such that the subgroup of totally positive units coincides with the subgroup of square of units. This idea has its origin in a paper of Shanks \cite{Sh}.}
\begin{equation}
\label{F10}
F_{1, 0} (X, Y) = X^3 -3XY^2 -Y^3.
\end{equation}
We check that $F_{1, 0}$ belongs to ${\rm Bin}(3, \Z)$ and that we have $F_{1, 0} \circ R = F_{1, 0}$. So we have the inclusion 
\begin{equation}
\label{inclusionAutC3}
{\mathbf C}_3 \subseteq {\rm Aut}(F_{1, 0}, \Q).
\end{equation} 
According to the inclusion \eqref{inclusionAutC3} and the list $\mathfrak H$ (see \eqref{frakH}), the group ${\rm Aut}(F_{1, 0}, \Q)$ can only be ${\rm GL}(2, \Q)$--conjugate to $\mathbf C_3$, $\mathbf C_6$, $\mathbf D_3$ or $\mathbf D_6$. Actually, \eqref{inclusionAutC3} may be refined to the equality
\begin{equation}
\label{Aut=C3}
{\rm Aut}(F_{1, 0}, \Q) = {\bf C}_3
\end{equation}
thanks to the following remark: any subgroup which is ${\rm GL}(2, \Q)$--conjugate with the groups ${\bf C}_6$, ${\bf D}_3$ and ${\bf D}_6$ contains an element of order two. But if ${\rm Aut}(F_{1, 0}, \Q)$ had an element $\gamma = \begin{pmatrix} a & b\\ c& d \end{pmatrix}$ with order two, the associated homography
$$
z \in \mathbb{P}^1(\mathbb C) \mapsto \mathfrak{h}(z) = \frac{az + b}{cz + d},
$$ 
considered as a bijection of the set of roots of the polynomial $F_{1, 0} (z,1) = z^3 -3z-1$, would satisfy the equality
$$
\mathfrak{h}(\zeta) = \zeta,
$$
for some root $\zeta$. So $\zeta$ would be rational or quadratic. But elementary considerations on the polynomial $F_{1, 0}(z, 1)$ show that all its roots have degree three. This gives a contradiction and \eqref{Aut=C3} is proved. 
 
We apply Corollary \ref{illustration2} with the natural choice $G = F_{1, 0}$. Let $H(X, Y) := G(2X, Y)$. By considering the discriminant, we see that the form $H$ is not ${\rm GL}(2, \Z)$--equivalent to $G$. Since $G = G \circ R = G \circ R^2$, we deduce, similarly to the proof of \eqref{G= G= G}, that the following equality
$$
G(\Z^2) = H(\Z^2)
$$
holds. By Corollary \ref{illustration2}, we deduce that $[F_{1, 0}]_{\rm val}$ breaks into exactly two equivalence classes under $\sim_{{\rm GL}(2, \Z)}$:
$$
[F_{1, 0}]_{\rm val} = [F_{1, 0}]_{{\rm GL}(2, \Z)} \cup [H]_{{\rm GL}(2, \Z)}.
$$
\end{example}

\begin{example} 
Consider now the cubic form $F(X, Y) = F_{2, 1}(X, Y) = 2X^3 +X^2Y -5XY^2 -2Y^3$. By the same type of remarks leading to \eqref{Aut=C3}, we deduce that $F$ is extraordinary. Starting from $F$, we build the two forms
$$
G(X, Y) := \frac{1}{2} F(X, Y),
$$
and 
$$
H(X, Y) := \frac{1}{2} F(2X, Y) = 8X^3 +2X^2Y -5 XY^2 -Y^3.
$$
Reasoning as in Example \ref{example1.7}, we deduce that the pair $(G,H)$ is an example of two extraordinary forms with equal images, but one of these forms has integer coefficients and the other one has a non-integral coefficient.
\end{example} 

\begin{example} 
The most studied binary form of degree $3$ is certainly the form $F(X, Y) = X^3 +Y^3$. We have ${\rm Aut}(F, \Q) = {\bf D}_1$, which is a group with two elements, see Lemma \ref{possibleAut}. By Corollary \ref{illustration1}, we deduce that the value set of a binary form is the set of sums of two cubes if and only if this form is ${\rm GL}(2, \Z)$--equivalent to $F$.
 
More generally, we find a complete description of the group ${\rm Aut}(aX^d + bY^d, \Q)$ in \cite[Lemma 3.3]{SX}, where $d \geq 3$ and $a$ and $b$ are integers of any sign with $ab \neq 0$. We check that the cardinality of this group is never divisible by $3$ and we deduce that the forms $aX^d + bY^d$ are always ordinary.
\end{example} 
 
\begin{example}
The last example in degree $d = 3$ concerns a perturbation of $F_{1, 0}$ (see \eqref{F10}) by the matrix $\rho = \begin{pmatrix} 4 & 0 \\0 & 1\end{pmatrix}$. Let 
$$ 
J(X, Y) := (F_{1, 0} \circ \rho) (X, Y) = 64 X^3 -12 XY^2 -Y^3.
$$
By the conjugacy property \eqref{conjugationformula} and the equality \eqref{Aut=C3}, we have
$$
{\rm Aut}(J, \Q) = \rho^{-1} {\bf C}_3 \rho = \left\{{\rm id}, \begin{pmatrix} 0 & 1/4 \\ -4 & -1\end{pmatrix}, \begin{pmatrix} -1 & -1/4 \\ 4 & 0 \end{pmatrix} \right\} =: \left\{ {\rm id}, \sigma, \sigma^2\right\}.
$$
Neither $\sigma$ nor $\sigma^2$ satisfies one of the conditions $\eqref{1}, \dots, \eqref{4}$ of Corollary \ref{illustration1}. Thus the form $J$ is ordinary.
\end{example}

It is natural to wonder if there are ``infinitely many different'' extraordinary forms. We make this precise in the following extension of Definition \ref{defextraordinary}.

\begin{definition}
Let $d \geq 3$. A form $F \in {\rm Bin}(d, \Q)$ is said to be {\it extraordinary} if there exists $G \in {\rm Bin}(d, \Q)$ satisfying $G(\Z^2) = F(\Z^2)$ and no $\gamma \in {\rm GL}(2, \Z) $ satisfying $G = F \circ \gamma$. The pair $(F, G)$ is called a pair of linked extraordinary forms.

Let $a\in \Q^*$, $\gamma \in {\rm GL}(2, \Z)$ and $F \in {\rm Bin}(d, \Q)$. Then the form 
$$
a(F \circ \gamma)
$$
is extraordinary if and only if $F$ is extraordinary. Let $(F, G)$ and $(F', G')$ be two pairs of linked extraordinary forms. These two pairs are called equivalent if there exists $a \in \Q^*$ and $\gamma, \delta \in {\rm GL}(2, \Z)$ such that 
$$ 
F' = a F \circ \gamma \textup{ and } G'= a G \circ \delta. 
$$
\end{definition} 

By \eqref{Del-Wat}, we know that there exist only two non-equivalent pairs of linked extraordinary forms, if one restricts to the set of definite (positive or negative) quadratic forms. As we prove below, this is not the case in the set of binary cubic forms.

\begin{corollary} 
There are infinitely many non-equivalent pairs of linked extraordinary forms in ${\rm Bin}(3, \Q)$.
\end{corollary} 
 
\begin{proof} 
We adopt the notations of Example \ref{example1.7} to consider the form
$$
\Phi_b(X, Y) := F_{1,b} (X, Y) = X^3 + b X^2 Y + (b - 3) XY^2 -Y^3,
$$
where $b$ is an integer. The discriminant of $\Phi_b$ is the polynomial
$$
{\rm disc}(\Phi_b) = b^2(b - 3)^2 -18b (b - 3) - 27 + 4b^3 - 4(b-3)^3 = (b^2 - 3b + 9)^2.
$$
% R.<x> = PolynomialRing(QQ)
% f = x^2 * (x - 3)^2 - 18 * x * (x - 3) - 27 + 4 * x^3 - 4 * (x - 3)^3
% f.factor()
Since ${\rm disc}(\Phi_b) \neq 0$, we have $\Phi_b \in {\rm Bin}(3, \Z)$. As it was proved in Example \ref{example1.7}, we have ${\rm Aut}(\Phi_b, \Q) ={\bf C}_3.$ Thanks to Corollary \ref{illustration2} we know that $(\Phi_b(X, Y), \Phi_b(2X, Y))$ is a pair of linked extraordinary forms.
 
Let $\gamma \in {\rm GL}(2, \Z)$ and $a \in \Q^*$ be such that $\Phi_{b'} = a \Phi_b \circ \gamma$. By using the transformation formulas of the discriminant, we have
$$
{\rm disc} (\Phi_{b'}) = a^4 \det(\gamma)^{3(3-1)} {\rm disc}(\Phi_b) = a^4 {\rm disc}(\Phi_b).
$$
By classical analytic number theory we may find infinitely many $b$ such that $b^2 - 3b + 9 > 1$ is squarefree. By the transformation formula, the resulting $\Phi_b$ for such $b$ can never be equivalent, and the corollary follows.
\end{proof} 

\subsection{The other cases} 
\label{othercases} 
In \cite{FKD4}, we will prove

\begin{thmx} 
\label{A} 
Let $d \geq 3$. There does not exist an extraordinary form $F \in {\rm Bin}(d, \Q)$ with ${\rm Aut}(F, \Q) \simeq_{{\rm GL}(2, \Q)}{\bf D}_4$.
\end{thmx}
 
Of course, Corollaries \ref{illustration1}, \ref{illustration2} and Theorem \ref{source} remain true if one replaces the condition $(C1)$ by $(C2)$ since they are empty. In \cite{FKD3D6}, we will prove 

\begin{thmx}
\label{B} 
Corollaries \ref{illustration1}, \ref{illustration2} and Theorem \ref{source} remain true if one replaces the condition $(C1)$ by $(C3)$.
\end{thmx}

\subsection*{Acknowledgements} 
The first author thanks Michel Waldschmidt for inspiring the thema of this paper, for sharing his ideas and for his encouragements. The second author gratefully acknowledges the support of Dr. Max R\"ossler, the Walter Haefner Foundation and the ETH Z\"urich Foundation.

\section{A first list of definitions and lemmas regarding binary forms} 
Let $F$ be a binary form with degree $d \geq 3$. The area of its fundamental domain is
$$
A_F := \iint_{\vert F (x, y) \vert \leq 1} dx dy.
$$
We have $0 < A_F < \infty$. Furthermore, for every $\gamma$ in ${\rm GL}(2, \mathbb R)$, we have the equality
\begin{equation}
\label{invarianceofarea}
A_{F \circ \gamma} = \vert \, \det \gamma\, \vert^{-1} \cdot A_F.
\end{equation}
Our first lemma bounds the number of solutions to $F(x, y) = m$.

\begin{lemma}
\label{EverBomSch} 
Let $d \geq 3$ and $F \in {\rm Bin}(d, \Z)$. Then for every $\varepsilon > 0$, there exists $c = c(F, \varepsilon)$ such that for every integer $m\neq 0$, the number of solutions in unknown integers $(x, y)$ to the equation $F(x, y) = m$ is less than $c \vert m\vert^\varepsilon.$
\end{lemma}

\begin{proof} 
The proof proceeds in several steps.
\vskip .3cm 
\noindent $\bullet$ If $F$ is irreducible and if we restrict to counting primitive solutions, which means with the restriction ${\rm gcd} (x, y)=1$, 
this is a direct consequence of \cite[Theorem 6.4]{E} and \cite[Theorem]{BS} containing improvements in the constant $c$. 
\vskip .3cm
\noindent $\bullet$ To drop the condition ${\rm gcd} (x, y) =1$, we use the homogeneity of $F$ to consider at most $\tau (\vert m \vert)$ equations of the following form
$$
F(x', y') = m /\kappa^d,
$$
where $\tau (n) $ is the number of positive divisors of the integer $n\geq 1 $, where $\kappa \geq 1$ is an integer satisfying $\kappa^d \mid m$ and where the unknowns $(x', y')$ now are coprime integers. Summing over all the possible $\kappa$ we obtain the result.
\vskip .3cm
\noindent $\bullet$ Suppose now that $F$ is reducible over $\Q$. We factorize it as $F_1= F_1 \cdots F_t$, where $t\geq 2$, where the $F_i$ are distinct and irreducible over $\Q$ and where $1 \leq \deg F_1 \leq \cdots \leq \deg F_t.$
\vskip .2cm
-- If $\deg F_t \geq 3$, the equality $F(x, y) =m$ implies that $F_t (x, y) =m'$, where $m'$ is some positive or negative divisor of $m$. Then apply the above part of the discussion to the equation $F_t (x, y) = m'$ and sum over all the possible $m' \mid m$ to get the required bound.
\vskip .2cm
-- If $\deg F_t \leq 2$, we necessarily have $t \geq 2$. The number of solutions to $F(x, y) =m$ is less than the total number of solutions to the two simultaneous equations $F_t( x, y) =m_1$ and $F_{t-1} (x, y) = m_2$, where the product $m_1m_2$ divides $m$. When $m_1m_2$ is fixed, we find at most four solutions to this pair of equations, since it corresponds to the intersection of two distinct conics, of a conic and a line, or of two distinct lines. Summing over all the possible pairs $(m_1,m_2)$ such that $m_1m_2\mid m$, we complete the proof. 
\end{proof}

Let $X \geq 1$ be a real number. We introduce the counting function
\begin{equation}
\label{defNF}
N(F, X) := \vert \{m: \vert m \vert \leq X, m = F(x_1, x_2) \text{ for some } (x_1, x_2) \in \Z^2\} \vert.
\end{equation}
We will use the following weak information on this function (for a precise study see \cite{SX}).

\begin{lemma}
\label{resultStewartXiaobis}
Let $d \geq 3$ and $F \in {\rm Bin}(d, \Z)$. Then, for every $\varepsilon > 0$, there exists $c(\varepsilon) > 0$ and $B_0(\varepsilon) \geq 1$, such that, for $B \geq B_0(\varepsilon)$ one has the inequality 
$$
N(F, B) \geq c(\varepsilon) B^{2/d-\varepsilon}.
$$
\end{lemma}

\begin{proof} 
Let $\eta := (\sum _{0 \leq i \leq d} \vert a_i\vert)^{-1}$ be the inverse of the sum of the absolute values of the coefficients of the form $F(X,Y) = \sum_i a_i X^i Y^{d-i}$. Then for $\vert x\vert, \vert y \vert \leq (\eta B)^{1/d}$, one has the inequality $\vert F(x, y) \vert \leq B$. So we exhibited a set $\mathcal E$ of $\gg B^{2/d}$ pairs of integers $(x, y)$ such that $\vert F (x, y) \vert \leq B$. Furthermore, we exclude from $\mathcal E$ all the pairs $(x, y)$ such that $F(x ,y) = 0$. This number is readily verified to be $O(B^{1/d})$, as they correspond to the rational linear factors of the form $F$.

Finally, to bound the number of distinct values of $F(x, y)$ we use Lemma \ref{EverBomSch}. It asserts that the number of solutions to the equation $F(x, y) = m$, where the unknowns are the integers $(x, y)$ and where $m \neq 0$ is a parameter, is bounded by $\ll_{\varepsilon} \vert m \vert^\varepsilon$. It remains to combine these informations to complete the proof.
\end{proof}

A result of Fouvry and Waldschmidt, see \cite[Theorem 1.1]{FW}, shows that the intersection of the images of two forms, not ${\rm GL}(2, \Q)$--equivalent, is much smaller than each of these images.

\begin{lemma} 
\label{intersectionofimages} 
For every $d \geq 3$, there exists a constant $\kappa_d <2/d$ such that for every $F, G \in {\rm Bin}(d, \Z)$, which are not ${\rm GL}(2, \Q)$--equivalent, and for every $N \geq 1$, we have the upper bound
$$
\left\vert\{m: \vert m \vert \leq N, m = F(x_1, x_2) = G(x_3, x_4) \textup{ for some } (x_1, x_2, x_3, x_4) \in \Z^4\}\right\vert = O_{F, G}\left(N^{\kappa_d} \right). 
$$
\end{lemma}

Combining Lemmas \ref{resultStewartXiaobis} and \ref{intersectionofimages}, we deduce 

\begin{lemma}
\label{existenceofrho} 
Let $F_1, F_2 \in {\rm Bin}(d, \Q)$ be such that $F_1(\Z^2) = F_2(\Z^2)$. Then there exists $\rho \in {\rm GL}(2, \Q)$ such that 
\begin{equation}
\label{F1rho= F2}
F_1 \circ \rho = F_2.
\end{equation}
\end{lemma}

This lemma answers positively the question \eqref{G= Fogamma}, if we impose the much weaker condition $\gamma \in {\rm GL}(2, \Q)$ instead of $\gamma \in{\rm GL}(2, \Z)$. 

\section{From Corollary \ref{illustration2} to Corollary \ref{illustration1}}
The goal of this section is to deduce Corollary \ref{illustration1} from Corollary \ref{illustration2}. We first recall some classical results classifying all the finite subgroups of ${\rm GL}(2, \Q)$ up to conjugacy. Then we prove Lemma \ref{433}, which plays a central role in our paper. This lemma gives sufficient conditions for a form to be extraordinary. The proof generalizes various constructions we have already seen in the introduction. We will eventually see that the conditions in Lemma \ref{433} are also necessary, although this result lies substantially deeper.

\subsection{A second list of definitions and lemmas} 
We generalize the notion of automorphism as follows.

\begin{definition}
Let $R \in \{\Z, \Q, \C\}$. Let $d \geq 3$ and let $F, G \in {\rm Bin}(d, R)$. An element $\rho \in {\rm GL}(2, R)$ is said to be an isomorphism from $F$ to $G$ if it satisfies the equality
$$
F \circ \rho = G. 
$$
The set of all isomorphisms from $F$ to $G$ is denoted by ${\rm Isom}(F \rightarrow G, R)$. We have
$$
{\rm Isom}(F\rightarrow F, \C) = {\rm Aut}(F, R).
$$ 
\end{definition}

The next lemma is obvious.

\begin{lemma} 
\label{959} 
Let $d \geq 3$ and let $F, G \in {\rm Bin}(d, \C)$. We have the equality 
$$
{\rm Isom}(G \rightarrow F, \C) = {\rm Isom}^{-1}(F \rightarrow G, \C),
$$ 
and the conjugation formula 
$$
{\rm Isom}(F \circ \gamma \rightarrow G\circ \delta , \C) = \gamma^{-1}{\rm Isom}(F\rightarrow G, \C) \delta
$$ 
for any $\gamma, \delta \in {\rm GL}(2, \C)$. 

Suppose that ${\rm Isom}(F\rightarrow G, \C)$ is not empty. Let $\rho$ be one of its elements. We then have the equality
$$
{\rm Isom}(F \rightarrow G, \C) = \rho \circ {\rm Aut}(G, \C) = {\rm Aut}(F, \C) \circ \rho.
$$
\end{lemma}

Theorem \ref{source} shows the crucial importance of the group ${\rm Aut}(F, \Q)$. As announced by \eqref{frakH}, the following lemma recalls all the possibilities for such a group. We adopt the notations of \cite[Table 1, p.137]{SX}.

\begin{lemma}
\label{possibleAut} 
Let $d \geq 3$ and let $F \in {\rm Bin}(d, \Q)$. Then ${\rm Aut}(F, \Q)$ is ${\rm GL}(2, \Q)$--conjugate to one of the ten following subgroups
of ${\rm GL}(2, \Q)$:

$$
\begin{matrix}
{\bf Subgroup}&{\bf Generators} & &{\bf Cardinality}\\
{\bf C}_1 :& \begin{pmatrix} 1 & 0\\0 & 1 \end{pmatrix} && 1&\\
{\bf C}_2: & \begin{pmatrix} -1& 0\\0 & -1 \end{pmatrix}& & 2&\\
{\bf C}_3 :& \begin{pmatrix} 0 & 1\\-1 & -1 \end{pmatrix}& & 3&\\
{\bf C}_4 :& \begin{pmatrix} 0 & 1\\-1 & 0 \end{pmatrix} && 4&\\
{\bf C}_6 :& \begin{pmatrix} 0 & -1\\1 & 1 \end{pmatrix} && 6&\\
{\bf D}_1: & \begin{pmatrix} 0 & 1\\1 & 0 \end{pmatrix} & &2&\\
{\bf D}_2: & \begin{pmatrix} 0 & 1\\1 & 0 \end{pmatrix} ,& \begin{pmatrix} -1& 0\\ 0 & -1 \end{pmatrix}& 4&\\
{\bf D}_3 :& \begin{pmatrix} 0 & 1\\1 & 0 \end{pmatrix} ,& \begin{pmatrix} 0& 1\\ -1 & -1 \end{pmatrix}& 6&\\
{\bf D}_4 :& \begin{pmatrix} 0 & 1\\1 & 0 \end{pmatrix} ,& \begin{pmatrix} 0 & 1\\ -1 & 0 \end{pmatrix}& 8&\\
{\bf D}_6 :& \begin{pmatrix} 0 & 1\\1 & 0 \end{pmatrix} ,& \begin{pmatrix} 0& 1\\ -1& 1 \end{pmatrix}& 12&\\
\end{matrix}
$$
This list can not be shortened. The set of these ten subgroups is denoted by $\mathfrak H$.
\end{lemma}
 
The two lemmas below will be used to deduce Corollary \ref{illustration2} from Theorem \ref{source} in \S \ref{section4}. Recall the notation \eqref{defR}.

\begin{lemma} 
\label{sigma<->C3}
Let $A_1$ be an integral matrix with order $3$ and let $\left\langle A_1\right\rangle$ be the subgroup generated by $A_1$. Then we have
$$
\left\langle A_1\right\rangle \simeq_{{\rm GL}(2, \Z)} {\mathbf C}_3.
$$
Furthermore, we have 
$$
A_1 \simeq_{{\rm GL}(2, \Z)} R.
$$
\end{lemma} 

\begin{proof} 
By \cite[p. 817, Table 1]{S}, we know that there exist exactly thirteen finite subgroups of ${\rm GL}(2, \Z)$, up to ${\rm GL}(2, \Z)$--conjugation (compare with Lemma \ref{possibleAut} which deals with ${\rm GL}(2, \Q)$--conjugation). This list contains only one subgroup with cardinality $3$. This is ${\bf C}_3$. The last statement follows from the fact that $R$ and $R^2$ are ${\rm GL}(2, \Z)$--conjugate by the matrix $\begin{pmatrix} 0 & 1\\ 1& 0\end{pmatrix}$.
\end{proof}

The following lemma gives some key properties of the binary forms with an automorphism of order $3$.

\begin{lemma}
\label{433}
Let $d \geq 3$ and let $F \in {\rm Bin}(d, \Z)$. Assume that ${\rm Aut}(F, \Z)$ contains an element $\sigma$ of order $3$. Then we have the following equivalences between forms 
$$ 
F(X, Y) \sim_{\rm val} F(2X, Y) \quad \textup{ and } \quad F(X, Y) \sim_{\rm val} F(X, 2Y).
$$
Furthermore, the forms $F(X, Y)$ and $F(2X, Y)$ are not ${\rm GL}(2, \Z)$--equivalent, but $F(2X, Y)$ and $F(X, 2Y)$ are ${\rm GL}(2, \Z)$--equivalent.
\end{lemma}

\begin{proof} 
Lemma \ref{sigma<->C3} implies that $\langle \sigma\rangle$ and ${\bf C}_3$ are ${\rm GL}(2, \Z)$--conjugate. Now let $\sigma$ be an element of ${\rm Aut}(F, \Z)$, which is ${\rm GL}(2, \Z)$--conjugate with $R$ defined in \eqref{defR}. Write $\sigma$ as
$$
\sigma:=\begin{pmatrix} a & b \\ c& d \end{pmatrix},
$$ 
and notice that $\det \sigma = 1$. Then for some integers $t_1$, $t_2$, $t_3$ and $t_4$ such that
\begin{equation}
\label{tt-tt=}
t_1t_4 -t_2 t_3 = \pm 1,
\end{equation}
one has the equality
\begin{equation}
\label{sigma=...}
\sigma = \pm \begin{pmatrix} t_4 & -t_2 \\ -t_3 & t_1 \end{pmatrix} \cdot R \cdot \begin{pmatrix} t_1 & t_2 \\ t_3 & t_4 \end{pmatrix} = \pm \begin{pmatrix} t_1t_2+t_2t_3+t_3t_4 & t_2^2 +t_4^2 +t_2t_4\\ -t_1^2 -t_3^2 -t_1t_3 & -t_1 t_2 -t_1 t_4 -t_3 t_4\end{pmatrix}.
\end{equation}
From \eqref{tt-tt=}, we see that
\begin{itemize}
\item $t_1$ and $t_3$ are not both even,
\item $t_2$ and $t_4$ are not both even,
\item $t_1, t_2, t_3, t_4$ are not all odd.
\end{itemize}
It follows from these facts and from \eqref{sigma=...} that $b$ and $c$ are odd and that precisely one of $a$ and $d$ is odd. Write $G_1(X, Y) := F(2X, Y)$ and $G_2(X, Y) := F(X, 2Y)$. We claim that
$$
F(X, Y) \sim_{\rm val} G_1(X, Y).
$$ 
One can similarly prove that $F(X, Y) \sim_{\rm val} G_2(X, Y)$. We visibly have $G_1(\Z^2) \subseteq F(\Z^2)$, so it suffices to prove the claim $F(\Z^2) \subseteq G_1(\Z^2)$. Suppose that the integers $m$ and $n$ are such that 
$$
F(m, n) =h.
$$
We generalize the transformations appearing in \eqref{G= G= G} as follows. Since $\sigma$ and $\sigma^{-1}$ are automorphisms of $F$, we have 
$$
F(m, n) = F(am + bn, cm + dn) = F(dm - bn, -cm + an).
$$
Because at least one of the three integers $m$, $am+bn$ and $dm-bn$ is even, we see that $h$ is also represented by $G_1$. This gives the claim. 

It remains to establish the last part of Lemma \ref{433}. Define
$$
\gamma := \begin{pmatrix} 2 & 0\\ 0 & 1\end{pmatrix},
$$ 
so that $F \circ \gamma = G_1$. As a consequence of \eqref{invarianceofarea}, the forms $F$ and $G_1$ are not ${\rm GL}(2, \Z)$--equivalent. We now show that $G_1$ and $G_2$ are ${\rm GL}(2, \Z)$--equivalent. We have $G_1 \circ \gamma' = G_2$, where
$$
\gamma' := \begin{pmatrix} 1/2 & 0 \\ 0 & 2 \end{pmatrix}.
$$
Let $\rho \in {\rm GL}(2, \Q)$. Then by Lemma \ref{959}, one has the equality 
\begin{equation}
\label{482}
G_1\circ \rho = G_2
\end{equation} 
if and only if
$$
\rho \in \{\gamma' \tau : \tau \in {\rm Aut}(G_2, \Q)\}.
$$
As a consequence of the definition of $G_2$ and of the conjugation formula \eqref{conjugationformula}, we have
$$
{\rm Aut}(G_2, \Q) = \begin{pmatrix} 1& 0\\ 0 & 2 \end{pmatrix}^{-1}\cdot {\rm Aut}(F, \Q) \cdot \begin{pmatrix} 1 &0\\ 0 & 2 \end{pmatrix}.
$$
Consider the element of order $3$
$$
\tau := \begin{pmatrix} 1 & 0 \\ 0 & 1/2\end{pmatrix} \cdot \sigma \cdot \begin{pmatrix} 1 & 0 \\0 & 2 \end{pmatrix} = \begin{pmatrix} a & 2b \\ c/2 & d \end{pmatrix},
$$
which is an element of ${\rm Aut}(G_2, \Q)$. We have the equalities
$$
\gamma' \tau = \begin{pmatrix} a/2 & b \\ c & 2d \end{pmatrix}, \quad \gamma' \tau^2 = \gamma' \tau^{-1} = \begin{pmatrix} d/2 &- b \\ -c & 2a \end{pmatrix}.
$$
Exactly one of these two matrices is integral. Call it $\rho$ and we obtain \eqref{482}, with $\rho \in {\rm GL}(2, \Z)$. Thus $G_1$ and $G_2$ are ${\rm GL}(2, \Z)$--equivalent, completing the proof of Lemma \ref{433}.
\end{proof}

We now have the tools to deduce Corollary \ref{illustration1} from Corollary \ref{illustration2}. We separate this proof in two parts.

\subsection{Deduction of Corollary \ref{illustration1} from Corollary \ref{illustration2}} 
\subsubsection{Proof of the {\it if} part} 
\label{theifpart} 
Let $F \in {\rm Bin}(d, \Z)$ be such that ${\rm Aut}(F, \Q)$ contains an element $\sigma$ of order $3$, written as in \eqref{writingsigma} and satisfying one of the conditions $\eqref{1}, \dots, \eqref{4}$. Let us prove that $F$ is extraordinary.

\begin{itemize}
\item If $\sigma$ satisfies \eqref{1}, it has integer entries. Lemma \ref{possibleAut} implies that 
$$
\{\tau \in {\rm Aut}(F, \Q) : \tau^3 = {\rm id}\} \subseteq {\rm GL}(2, \Z)
$$
and that $3 \mid |{\rm Aut}(F, \Q)|$. We now apply Corollary \ref{illustration2} with $G = F$.
\item Suppose that $\sigma $ satisfies \eqref{2} or \eqref{3}. Since $\det \sigma = 1$, we see that $bc$ must be an integer and therefore exactly one of the integers $b$ or $c$ is an even integer. Then exactly one of the following matrices 
$$
\begin{pmatrix}
1 &0\\ 0 & 2\end{pmatrix} \cdot \begin{pmatrix} a&b\\ c& d \end{pmatrix} \cdot \begin{pmatrix}
1 &0\\ 0 & 1/2\end{pmatrix} 
\text { or } 
\begin{pmatrix}
2 &0\\ 0 & 1\end{pmatrix} \cdot \begin{pmatrix} a&b\\ c& d \end{pmatrix} \cdot \begin{pmatrix}
1/2 &0\\ 0 & 1 \end{pmatrix}
$$
is an integral matrix. In the former case we apply Lemma \ref{433} to the form
\begin{equation}
\label{dag}
F^\dag := F \circ \begin{pmatrix} 1&0\\ 0 &1/2\end{pmatrix}.
\end{equation}
By \eqref{conjugationformula} we have
\begin{equation}
\label{534}
\{ {\rm id} \} \subsetneq\{\tau \in {\rm Aut}(F^\dag, \Q) : \tau^3 = {\rm id}\} \subseteq {\rm GL}(2, \Z).
\end{equation}
In the latter case we apply Lemma \ref{433} to 
\begin{equation}
\label{ddag}
F^\ddag := F \circ \begin{pmatrix} 1/2 & 0 \\ 0 & 1\end{pmatrix},
\end{equation}
which satisfies a property similar to \eqref{534}.

Let us concentrate on the case \eqref{dag}. Since ${\rm Aut}(F^\dag, \Z)$ contains an element of order $3$, we apply Lemma \ref{433} to $F^\dag$. We obtain that
$$
F^\dag (X, Y) \sim_{\rm val} F^\dag (X, 2Y) = F(X, Y),
$$ 
and that $F^\dag \not \sim_{{\rm GL}(2, \Z)} F$. The case \eqref{ddag} is similar. In both cases we conclude that $F$ is extraordinary.
\item Suppose that $\sigma$ satisfies \eqref{4}. We will reduce to the previous case by now considering the form $F^*$ defined by
$$
F^* = F \circ \begin{pmatrix} 1 & 1\\ 0& 1 \end{pmatrix}.
$$
The form $F^*$ is ${\rm GL}(2, \Z)$--equivalent to $F$ and its group of automorphisms contains the following element of order $3$
$$
\begin{pmatrix} 
1 & -1\\ 0 & 1 
\end{pmatrix}
\cdot
\begin{pmatrix} 
a & b \\ c & d 
\end{pmatrix}
\cdot 
\begin{pmatrix} 
1 & 1 \\ 0& 1
\end{pmatrix}
=
\begin{pmatrix}
a-c & a + b - c - d \\ c & c + d
\end{pmatrix}.
$$
This element satisfies \eqref{2}. So we are led to the previous item.
\end{itemize}

\subsubsection{Proof of the {\it only if} part} 
Let $F$ be an extraordinary form. By Corollary \ref{illustration2} we associate a form $G\in [F]_{\rm val}$ satisfying \eqref{268} and \eqref{269}. Remark that if $G$ satisfies these properties, then so does $G \circ \alpha$ for $\alpha \in {\rm GL}(2, \Z)$. Therefore, thanks to Lemma 
\ref{sigma<->C3}, we may assume that $R\in {\rm Aut}(G, \Z)$. Then by \eqref{269} we have
\begin{equation}
\label{twocases}
\text{ either } \ \ F = G \circ \gamma \ \text{ or } \ F = G \circ \begin{pmatrix} 2 & 0 \\ 0 & 1 \end{pmatrix} \circ \gamma
\end{equation}
for some $\gamma \in {\rm GL}(2, \Z)$. We will use the conjugation property \eqref{conjugationformula}. 

\begin{itemize}
\item If $G = F \circ \gamma$, it is clear that ${\rm Aut}(F, \Z)$ contains $ \sigma := \gamma^{-1} R \gamma $. This element $\sigma$ has order 3, has integer coefficient and satisfies \eqref{1}. 
\item Now suppose that we are in the second case of \eqref{twocases}. Since $R$ belongs to ${\rm Aut}(G, \Q)$, a non trivial element 
of ${\rm Aut}(F, \Q)$ is
$$
\gamma^{-1} \circ \begin{pmatrix} 1/2 & 0 \\ 0 & 1 \end{pmatrix} \circ R \circ \begin{pmatrix} 2 & 0 \\ 0 & 1 \end{pmatrix} \circ \gamma =\gamma^{-1} \circ \begin{pmatrix} 0 & 1/2 \\ -2 & -1 \end{pmatrix} \circ \gamma.
$$
Writing
$$
\gamma := \begin{pmatrix} a_1 & a_2 \\ a_3 & a_4 \end{pmatrix},
$$
where the $a_i$ are integers such that $a_1a_4 -a_2 a_3 = \pm1$, we get
\begin{align*} 
\gamma^{-1} \circ \begin{pmatrix} 
0 & 1/2 \\ -2 & -1 \end{pmatrix} \circ \gamma &= \pm \begin{pmatrix} a_4 & -a_2 \\ -a_3 & a_1\end{pmatrix}\cdot 
\begin{pmatrix} 
0 & 1/2 \\ -2 & -1 \end{pmatrix} \cdot \begin{pmatrix} a_1 & a_2 \\ a_3 & a_4 \end{pmatrix}\\
& = \pm \begin{pmatrix} 2a_1a_2 +a_3 (a_2 +a_4/2) & 2 a_2^2 +a_4 (a_2 +a_4/2) \\
-2a_1^2 -a_3 (a_1 +a_3/2)
 & -2a_1a_2 -a_4 (a_1+a_3/2)
\end{pmatrix}.
\end{align*}
Since $a_3$ and $a_4$ are not simultaneously even, the above matrix corresponds to the cases \eqref{2}, \eqref{3} or \eqref{4} of Corollary \ref{illustration1}. This proves the {\it only if} part.
\end{itemize}

\noindent The proof of Corollary \ref{illustration1} is complete assuming the truth of Corollary \ref{illustration2}.

\section{From Theorem \ref{source} to Corollary \ref{illustration2}}
\label{section4}
We adopt the hypotheses of Corollary \ref{illustration2}. So let $F \in {\rm Bin}(d, \Z)$ with $d \geq 3$ satisfy $(C1)$. Let us first show that $[F]_{\rm val}$ breaks into at most two equivalence classes under $\sim_{{\rm GL}(2, \Z)}$. Suppose for the sake of contradiction that $[F]_{\rm val}$ breaks into at least three equivalence classes. Let $F_1$, $F_2$ and $F_3$ be three elements in $[F]_{\rm val}$, that are pairwise not ${\rm GL}(2, \Z)$--equivalent. Since they have the same image, they are ${\rm GL}(2, \Q)$--equivalent by Lemma \ref{existenceofrho}. So write
$$
F_2 = F_1\circ \gamma_1, \quad F_3= F_2 \circ \gamma_2
$$
with $\gamma_1, \gamma_2 \in {\rm GL}(2, \Q)$. Applying Theorem \ref{source} to the forms $F_1$ and $F_2$, we obtain that $\vert \det \gamma_1 \vert \in \{1/2, 2\}$. Similarly, we obtain $\vert \det \gamma_2 \vert \in \{1/2, 2\}$. However, the equality $F_3 = F_1 \circ (\gamma_1 \cdot \gamma_2)$ also implies that 
$$
\vert \det (\gamma_1 \cdot \gamma_2) \vert = \vert \det \gamma_1 \vert \cdot \vert \det \gamma_2 \vert \in \{1/2, 2\},
$$ 
which contradicts the fact that both $\vert \det \gamma_1 \vert$ and $\vert \det \gamma_2 \vert$ belong to $\{1/2, 2\}$. 

We now claim that $[F]_{\rm val}$ breaks into two equivalence classes if and only if there exists $G \in [F]_{\rm val}$ such that $\{\sigma \in{\rm Aut}(G, \Q) : \sigma^3 = {\rm id}\} \subseteq {\rm GL}(2, \Z)$ and $3 \mid \vert {\rm Aut}(G, \Q)\vert$. The {\it if } part is an immediate consequence of Lemma \ref{433}, applied to the form $G$. The {\it only if} part is a consequence of Theorem \ref{source}.

The final part of Corollary \ref{illustration2} is a consequence of Lemma \ref{433}.

\section{Strategy of the proof of Theorem \ref{source}}
\label{strategy} 
We now pass to the proof of Theorem \ref{source}. This proof is quite complicated, so we take some time to present its main steps.

We start from two forms $F_1$ and $F_2$ satisfying \eqref{hypoforF1F2}, i.e. $F_1$ and $F_2$ have the same image but are not ${\rm GL}(2, \Z)$--equivalent. By Lemma \ref{existenceofrho}, there exists some $\rho \in {\rm GL}(2, \Q) \setminus {\rm GL}(2, \Z)$ such that $F_1\circ \rho = F_2. $ This $\rho$ is not unique, however we choose a $\rho$ or $\rho^{-1}$ , which, in some sense, is economical in terms of denominators (see \eqref{minimalindex}). Bringing $\rho$ into Smith normal form leads to the existence of two forms $G_1$ and $G_2$ respectively ${\rm GL}(2, \Z)$--equivalent to $F_1$ and $F_2$ and integers $D$ and $\nu$ such that $G_2(DX, DY)= G_1(X,\nu Y)$ (see \eqref{G1G2Dnu}). We are now dealing with integers only. By construction $G_1$ and $G_2$ have the same image, but they are not ${\rm GL}(2, \Z)$--equivalent. This implies the conditions \eqref{conditionsonDnu} on $D$ and $\nu$.

It is now natural to consider not only the projective surface $\mathbb W (G_1 \cap G_2)$ defined by the equation $G_1(X, Y) - G_2(Z, T) = 0$, but also the set of integers
$$
\{m \in \Z : m = G_1(x, y) = G_2(z, t) \text{ for } (x, y, z, t) \in \Z^4\}. 
$$
Actually, we will prove that most of the $m$ come from points situated on lines of the surface $\mathbb W (G_1 \cap G_2)$. Again by a counting process based on the geometry of varieties over $\Q$, we exhibit a non--trivial covering $\mathcal C$ of $\Z^2$ by lattices corresponding to the lines mentioned above. 

These lines are defined by the isomorphisms between $G_1$ and $G_2$ (see Proposition \ref{Bois-Sarti}). The equality $G_2(DX, DY)= G_1(X,\nu Y)$ produces one such obvious isomorphism. We obtain all the other by composing with automorphisms of $G_1$ or $G_2$. Proposition \ref{proposition6.1} is crucial by giving an explicit covering $\mathcal{C}$ of $\Z^2$ by lattices. For the moment we did not use the information on the structure of the automorphism groups of the forms $F_1$ and $F_2$. Recall that we have
$$
{\rm Aut}(G_i, \Q) \simeq_{{\rm GL}(2, \Z)} {\rm Aut}(F_i, \Q) \quad (i \in \{1, 2\}) \quad \text{ and } \quad {\rm Aut}(F_1, \Q) \simeq_{{\rm GL}(2, \Q)} {\rm Aut}(F_2, \Q).
$$
The end of the proof depends on the complexity of $\mathcal C$. Indeed, its cardinality (number of lattices belonging to $\mathcal C$) is at most $\vert {\rm Aut}(F_1, \Q)\vert /2$ if $-{\rm id}$ belongs to ${\rm Aut}(F_1, \Q)$, and $\vert {\rm Aut}(F_1, \Q)\vert $ otherwise (see Lemma \ref{gamma-gamma}). By Lemma \ref{possibleAut} this cardinality can take the values $1$, $2$, $3$, $4$ and $6$. In this paper, we are concerned with the cardinalities $1$, $2$ and $3$, corresponding to the condition $(C1)$. The condition $(C1)$ avoids substantial difficulties coming from the complexity of the covering $\mathcal C$, since there is no non--trivial covering $\mathcal C$ with cardinality $1$ or $2$. For cardinality $3$, there is only one possibility for $\mathcal C$ (see Lemmas \ref{2lattices} \& \ref{3lattices}).

When the cardinality of $\mathcal C$ is $4$, the situation is more complicated, since there are four possibilities for $\mathcal C$ with cardinality $4$ (proving this in itself is already a non-trivial task). It corresponds to the condition $(C2)$ and it will be treated in \cite{FKD4} by proving that there is no extraordinary form satisfying $(C2)$. This gives Theorem \ref{A}, which will be proved in \cite{FKD4}.

Finally, when the cardinality of $\mathcal C$ is at most $6$, it becomes infeasible to list all possible coverings without the help of a computer. Indeed, there are 9 possibilities for $\mathcal C$ to have a cardinality equal to $5$ and $49$ possibilities for cardinality $6$. The argument splits in many different cases and also requires further assistance from the computer. This is a very delicate part based on the explicit equations of the six lattices of $\mathcal C$ and this leads to Theorem \ref{B}, which will be proved in \cite{FKD3D6}.

\section{A third list of lemmas}
\subsection{About lattices}
In our work, a {\it lattice} means a subgroup of $\Z^2$ with rank two. Let $\Lambda_1, \dots, \Lambda_k$, be $k$ lattices. We say that the lattices $(\Lambda_i)_{1\leq i \leq k} $ are a {\it covering} of $\Z^2$ when their union equals $\Z^2$. A finite intersection of lattices is also a lattice. A lattice $\Lambda$ is a {\it proper lattice} if and only if $\Lambda \neq \Z^2$. The {\it index} of $\Lambda$ is the positive integer 
$$
[\Z^2 : \Lambda] = \vert \Z^2 / \Lambda \vert,
$$
and a lattice is proper if and only if its index is different from one. If $(\vec u, \vec v)$ is a $\Z$--basis of $\Lambda$, then we have 
\begin{equation}
\label{index=det}
[\Z^2 : \Lambda] = \vert \det (\vec u, \vec v) \vert.
\end{equation}

\begin{definition}
\label{associatedlattice} 
Let $\gamma\in {\rm GL}(2, \Q)$. The lattice $L(\gamma)$ associated with $\gamma$ is by definition
$$
L(\gamma) = \left\{(u, v) \in \Z^2 : \gamma(u, v) \in \Z^2 \right\}.
$$ 
\end{definition}

Since the elements of ${\rm GL}(2, \Z)$ are bijections on $\Z^2$ and preserve the index of lattices, we have

\begin{lemma}
\label{conservationofindex}
For every $\gamma \in {\rm GL}(2, \Q)$ and for every $P, Q \in {\rm GL}(2, \Z)$, one has the equalities 
$$ 
L(P\gamma) = L(\gamma), \quad L(\gamma Q) = Q^{-1} (L(\gamma)) \ \ \textup{ and } \ \ \left[\Z^2 : L(P \gamma Q)\right] = \left[\Z^2 : L(\gamma)\right].
$$
\end{lemma}

We now give several lemmas. We will omit the proofs.

\begin{lemma}
\label{lemma5.4} 
Let $\gamma \in {\rm GL}(2, \Q)$. Then
\begin{itemize}
\item the equality $L(\gamma) = \Z^2$ holds if and only if all entries of $\gamma$ are integers,
\item the equalities $L(\gamma) = L(\gamma^{-1}) = \Z^2$ hold if and only if $\gamma$ belongs to ${\rm GL}(2, \Z)$.
\end{itemize}
\end{lemma}

\begin{lemma}
\label{gamma-gamma}
We have $L(- \gamma) = L(\gamma)$ for every $\gamma \in {\rm GL}(2, \Q)$.
\end{lemma}

\begin{lemma}
\label{existenceofC} 
For every lattice $\Lambda \subseteq \Z^2$ there exists an integer $C \geq 1$ such that $C \Z^2 \subseteq \Lambda$.
\end{lemma}

\begin{proof} 
There exists a $\Z$--basis $(\vec u, \vec v)$ of $ \Z^2$ and positive integers $d_1$ and $d_2$ such that $(d_1 \vec u, d_2 \vec v)$ is a $\Z$--basis of $\Lambda$. Choose $C = d_1d_2$ for instance. 
\end{proof}

\subsection{\texorpdfstring{Covering $\Z^2$ with lattices}{Covering the two-dimensional lattice}}
We are now concerned with coverings made with two or three lattices.

\begin{lemma}
\label{2lattices}
Let $(\Lambda_i)_{1\leq i \leq 2}$ be a covering. Then at least one of these lattices is not proper.
\end{lemma}

\begin{proof}
Consider the three points $(1, 0)$, $(0,1)$ and $(1,1)$. Then at least one of the $\Lambda_i$, say $\Lambda_1$, contains at least two of these three points. Then for any choice, these two points generate $\Z^2$. Thus $\Lambda_1$ is not proper.
\end{proof}

We now study the covering of $\Z^2$ by three lattices.

\begin{lemma}
\label{3lattices} 
Let $\Lambda_0$, $\Lambda_1$ and $\Lambda_2 \subset \Z^2$ be three proper lattices such that $(\Lambda_i)_{0\leq i \leq 2}$ is a covering. Then, up to permutation, we have
\begin{align*}
\Lambda_0 &= \{(u, v) \in \Z^2 : u \equiv 0 \bmod 2\}, \\
\Lambda_1 &= \{(u, v) \in \Z^2 : v \equiv 0 \bmod 2\}, \\
\Lambda_2 &= \{(u, v) \in \Z^2 : u +v \equiv 0 \bmod 2\}.
\end{align*}
\end{lemma}

\begin{proof} 
The proof is similar to the above proof with the difference that we consider four points. Let $\Lambda_0, \Lambda_1, \Lambda_2$ be three proper sublattices satisfying 
\begin{equation*}
\Lambda_0 \cup \Lambda_1 \cup \Lambda_2 = \Z^2.
\end{equation*}
Define the following four points in $\Z^2$
$$
P_1 = (1, -1), \quad P_2 = (1, 0), \quad P_3 = (1, 1), \quad P_4 = (0, 1).
$$
By appealing to \eqref{index=det}, we check that for all distinct $i, j \in \{1, \dots, 4\}$
\begin{equation}
\label{735}
P_i\Z \oplus P_j\Z = \Z^2 \Longleftrightarrow \{i, j\} \neq \{1, 3\}.
\end{equation}
By the pigeonhole principle one of the $\Lambda_i$ must contain at least two $P_k$. Up to renumbering we may assume that $\Lambda_2$ contains at least two $P_k$. But since $\Lambda_2$ is a proper sublattice, this forces 
$$
\Lambda_2 = P_1\Z \oplus P_3\Z = \{(u, v) \in \Z^2 : u + v \equiv 0 \bmod 2\}.
$$
By \eqref{735} we also get that $(0, 1) \in \Lambda_0$ and $(1, 0) \in \Lambda_1$, after a potential renumbering.

Now look at the points $(2, 1)$ and $(1, 2)$. The point $(2, 1)$ is not in $\Lambda_2$, so it must be that $\Lambda_0$ or $\Lambda_1$ contains $(2, 1)$. But since $\Lambda_1$ is a proper sublattice containing $(1, 0)$, we see that $(2, 1) \not \in \Lambda_1$. Therefore we conclude that $(2, 1) \in \Lambda_0$ and similarly $(1, 2) \in \Lambda_1$. This gives the lemma.
\end{proof}
 
The description of the coverings of $\Z^2$ by four, five or six lattices is much more intricate. It will be given in \cite{FKD4} and \cite{FKD3D6}. See also the end of \S \ref{strategy} for more details.

\subsection{About matrices}
The first lemma gives a {\it canonical } way to express an element of ${\rm GL}(2, \Q)$ in terms of integers.

\begin{lemma}
\label{canonicalmatrix} 
Every matrix $M \in {\rm GL}(2, \Q)$ can be uniquely written as
$$
M = \frac{N}{D} \begin{pmatrix} m_1 & m_2 \\ m_3 & m_4 \end{pmatrix},
$$
where the integers $m_1$, $m_2$, $m_3$, $m_4$, $N$ and $D$ satisfy the conditions
$$
N, D \geq 1, \quad \gcd(N, D) = 1, \quad \gcd(m_1, m_2, m_3, m_4) = 1.
$$
In particular when $\det M = \pm 1$, one has $N = 1$.
\end{lemma}

\begin{proof}
This is straightforward.
\end{proof}

We use the canonical form to compute the index of the lattice from Definition \ref{associatedlattice}:

\begin{lemma}
\label{lemma5.10}
Let $f \in {\rm GL}(2, \Q)$, whose associated matrix $M_f$ has canonical form given by
$$
M_f := \frac{N}{D}
\begin{pmatrix} 
m_1& m_2 \\ m_3 & m_4
\end{pmatrix}.
$$
Then the lattice $L(f)$ satisfies the equality
$$
[\Z^2 : L (f)] = \frac {D^2}{\gcd(D, m_1m_4 - m_2m_3)}.
$$
Finally, $[\Z^2 : L(f)]$ is an integral multiple of $\vert \det f \vert^{-1}$.
\end{lemma}

\begin{proof} 
Use the Smith normal form (see \cite[Theorem 2.4.12]{Coh} for instance) to write the matrix $M_f$ as
$$
M_f = \frac ND \cdot U \cdot \begin{pmatrix} 1 & 0 \\ 0 & \pm (m_1m_4-m_2m_3)\end{pmatrix} \cdot V,
$$
where $U$ and $V$ belong to ${\rm GL}(2, \Z)$. Then by Lemma \ref{conservationofindex} we have
$$
[\Z^2 : L(f)] = [\Z^2 : L],
$$ 
where $L$ is the lattice
$$
L:= \{(x, y) \in \Z^2 : \bigl((N/D) x, (N/D)(m_1m_4-m_2m_3)y\bigr) \in \Z^2\}.
$$
Since $\gcd(N, D) = 1$, the index of $L$ equals $D \cdot D/\gcd(D, m_1m_4 - m_2m_3)$. For the last part, notice that our formula implies that $[\Z^2 : L(f)]$ is an integral multiple of $D^2/(m_1m_4 - m_2m_3)$, which is a multiple of $D^2/(N^2 (m_1m_4-m_2m_3)) = 1/ \det f$.
\end{proof}
 
\subsection{About values of polynomials and application to binary forms}
The following lemma has appeared in various places in the literature. Bhargava attributes this result to P\' olya \cite{Pol}.

\begin{lemma}
\label{Bhar} 
Let $f (X)\in \Z [X]$ be a primitive polynomial of degree $k \geq 0$. Then
$$
\gcd(\{f(a) : a \in \Z\}) \mid k!.
$$
\end{lemma}

\begin{proof} 
This statement is trivial when $k=0$. For $k \geq 1$, see \cite [Theorem 2]{Bhargava} for a proof. 
\end{proof}

We apply this result to situations concerning binary forms.

\begin{lemma}
\label{Lemma5.6} 
Let $d \geq 1$ be an integer and let $H(X, Y)$ be a binary form with integer coefficients and with degree $d$. Suppose that there exists positive integers $D$, $N$ and $\nu$ such that $\gcd(D, N) =1$ and such that
\begin{equation}
\label{713}
D^d H(\Z^2) = \left[N^d H \circ \begin{pmatrix} 1 & 0 \\ 0 &\nu \end{pmatrix}\right] (\Z^2).
\end{equation}
Then we have $N=1$ and $D \mid \nu$.
\end{lemma}

\begin{proof} 
By dividing $H$ by the gcd of its coefficients, we may assume that $H$ is primitive. Since $D$ and $N$ are coprime, we deduce from the hypothesis that 
$$
H(m, n) \equiv 0 \bmod N^d
$$
for all integers $m$ and $n$. Assume for the sake of contradiction that $N >1$ and let $p$ be a prime dividing $N$. Consider the polynomial $P(T) := H(T, 1)$, which is a primitive polynomial in $\Z[T]$ with degree $d'$ satisfying $0 \leq d' \leq d$. By definition of $p$ we have 
$$
P(t) = H(t, 1) \equiv 0 \bmod p^d
$$ 
for every integer $t$. Applying Lemma \ref{Bhar}, we deduce that $p^d \mid d'! \mid d!$. Such a divisibility property never happens since, by denoting the $p$--adic valuation by $v_p$, we have the inequality $v_p(d!) < d$, for every prime $p$ and every integer $d \geq 1$. The hypothesis $N > 1$ leads to a contradiction. Therefore we have $N = 1$. 
 
To prove that $D \mid \nu$, we start from \eqref{713} with $N = 1$. For an integer $t$, consider the values
$$
\left[H \circ \begin{pmatrix} 1 & 0 \\ 0 & \nu \end{pmatrix}\right](\nu, t) = H(\nu, \nu t) = \nu^d H(1, t).
$$
We deduce from \eqref{713}, that for every integer $t$, one has
$$
D^d \mid \nu^d H(1, t),
$$
that we write as 
$$
\left(\frac{D}{\gcd(\nu, D)}\right)^d \mid H(1, t).
$$
Again appealing to Lemma \ref{Bhar}, we deduce that $D/\gcd(\nu, D) = 1$. Hence the result.
\end{proof}
 
\section{\texorpdfstring{From $(F_1, F_2)$ to $(G_1, G_2)$}{A reduction step}}
\label{FromF1F2toG1G2}
\subsection{\texorpdfstring{The endomorphism $\rho$}{The endomorphism rho}}
\label{constructionofrho} 
Assume the hypotheses of Theorem \ref{source}: we start from two forms $F_1$ and $F_2$ satisfying \eqref{hypoforF1F2}. Furthermore, we suppose that $F_1$ satisfies $(C1)$. By Lemma \ref{existenceofrho}, there exists $\rho \in {\rm GL}(2, \Q)$ satisfying \eqref{F1rho= F2}. By \eqref{conjugationformula} the form $F_2$ also satisfies $(C1)$. This $\rho$ is not necessarily unique. To choose a particular $\rho$, we consider the finite set 
$$
\Gamma := {\rm Isom}(F_1\rightarrow F_2, \Q) \cup {\rm Isom}(F_2 \rightarrow F_1, \Q), 
$$
we recall the Definition \ref{associatedlattice} and we pick some $\rho \in \Gamma$ such that 
\begin{equation}
\label{minimalindex}
[\Z^2 : L (\rho^{-1})] \text{ is minimal among } \{ [\Z^2 : L (\tau)] : \tau \in \Gamma\}.
\end{equation} 
Swapping the roles of $F_1$ and $F_2$ if necessary, we may still assume that 
$$
F_1\circ \rho = F_2,
$$
where $\rho \in {\rm GL}(2, \Q)$ now satisfies \eqref{minimalindex}.
 Using Lemma \ref{canonicalmatrix}, we write $\rho$ in its canonical form
$$
\rho = \frac ND \cdot A,
$$
where $A = \begin{pmatrix} a_1 & a_2 \\ a_3 & a_4 \end{pmatrix}$. From homogeneity it follows that
\begin{equation}\label{NdF1A=}
N^dF_1 \circ A = D^d F_2.
\end{equation}
The coefficients of the matrix $A$ are coprime integers and the determinant of $A$ is different from zero. Writing $A$ in Smith normal form gives the decomposition
$$
A= P \begin{pmatrix} 1 & 0 \\ 0 & \nu\end{pmatrix} Q,
$$
where $P$, $Q \in {\rm GL}(2, \Z)$ and $\vert \nu \vert = \vert \det A \vert = \vert a_1 a_4-a_2 a_3 \vert$. Furthermore, we can suppose that 
$$
\nu \geq 1,
$$
by eventually multiplying, on the left, the matrix $Q$ by the matrix $\begin{pmatrix} 1 &0 \\ 0& -1\end{pmatrix}$. By assumption we have
$$
F_1(\Z^2) = F_2(\Z^2).
$$
Combining with \eqref{NdF1A=} we obtain the equalities 
\begin{align*}
D^d F_1(\Z^2) = D^d F_2(\Z^2) & =\left[N^d F_1 \circ A \right](\Z^2) = \left[ N^d F_1 \circ P \circ \begin{pmatrix} 1 & 0 \\ 0 & \nu \end{pmatrix} \circ Q \right] (\Z^2) \\
& = \left[ N^d F_1 \circ P \circ \begin{pmatrix} 1 & 0\\ 0 & \nu\end{pmatrix}\right](\Z^2),
\end{align*}
since $Q$ permutes $\Z^2$. Since $P^{-1} $ also permutes $\Z^2$, we deduce the equality of images
$$
D^d G (\Z^2) = \left[ N^d G \circ \begin{pmatrix} 1 & 0 \\ 0 & \nu\end{pmatrix} \right] (\Z^2),
$$
with $G = F_1 \circ P$. We apply Lemma \ref{Lemma5.6} to obtain that 
$$
N= 1 \text{ and } D \mid \nu.
$$ 
This simplifies \eqref{NdF1A=} into
$$
D^d F_2 = F_1 \circ P \circ \begin{pmatrix} 1 & 0 \\ 0 & \nu\end{pmatrix} \circ Q.
$$
Let
$$
G_1 := F_1 \circ P \text{ and } G_2:= F_2 \circ Q^{-1},
$$
so we have 
\begin{equation}
\label{music}
F_i \sim_{{\rm GL}(2, \Z)} G_i.
\end{equation}
Furthermore, these binary forms satisfy the following
\begin{equation}
\label{G1G2Dnu}
\begin{cases}
G_2(DX, DY) = G_1(X, \nu Y) \\
G_1 \sim_{\rm val} G_2 \\
G_1 \not\sim_{{\rm GL}(2, \Z)} G_2,
\end{cases}
\end{equation}
where $D$ and $\nu$ are such that 
\begin{equation}
\label{conditionforDandnu}
D, \nu \geq 1, D \nu > 1 \text{ and } D \mid \nu.
\end{equation}
The forms $G_1$ and $G_2$ are those which appear in Theorem \ref{source}. We now have to prove that $3 \mid \vert {\rm Aut}(F_1, \Q) \vert$, which equals $\vert {\rm Aut}(F_2, \Q)\vert = \vert {\rm Aut}(G_1, \Q)\vert = \vert {\rm Aut}(G_2, \Q)\vert$, and that $(D, \nu) \in \{(1, 2), (2, 2)\}$. 

\subsection{\texorpdfstring{Further remarks on $D$ and $\nu$}{Further remarks on D and nu}} 
Of particular importance is the matrix
\begin{equation}
\label{DD/nu}
\gamma:= \begin{pmatrix} D & 0 \\ 0 & D/\nu\end{pmatrix}.
\end{equation}
With the notations of \S \ref{constructionofrho}, we have the relations
\begin{equation*}
\gamma = Q\, \rho^{-1} P,
\end{equation*} 
and 
\begin{equation}
\label{G1= G2ogamma}
G_1 = G_2 \circ \gamma.
\end{equation}
Thus the element $\gamma$ belongs to ${\rm Isom} (G_2\rightarrow G_1 , \Q)$. By Lemmas \ref{959} \& \ref{conservationofindex} and by the definition \eqref{minimalindex} of $\rho$, we infer that 
\begin{equation}
\label{minimalindexforgamma}
L(\gamma) \text{ has minimal index among } L(\tau) \text{ for } \tau \in {\rm Isom}(G_1 \rightarrow G_2, \Q) \cup {\rm Isom}(G_2 \rightarrow G_1, \Q).
\end{equation}
In particular, we have the inequality
\begin{equation}
\label{index<index}
\left[ \Z^2 : L (\gamma) \right] \leq \left[ \Z^2 : L(\gamma^{-1})\right].
\end{equation}
The lattices $L(\gamma)$ and $L(\gamma^{-1})$ are respectively the sets of $(x_1, x_2)\in \Z^2$ such that
$$
\begin{cases} Dx_1 \in \Z\\
D\nu^{-1} x_2 \in \Z
\end{cases}
\text{ and }
\begin{cases}
D^{-1} x_1 \in \Z\\
\nu D^{-1} x_2 \in \Z
\end{cases}.
$$
The corresponding indices are
$$
 \left[ \Z^2 : L (\gamma) \right] = \frac \nu {{\rm gcd} (D, \nu)}
\text{ and }
\left[ \Z^2 : L (\gamma^{-1}) \right] = D\cdot \frac D {{\rm gcd} (D, \nu)}.
$$
By \eqref{index<index} we obtain the inequality $1\leq \nu \leq D^2.$ By the condition \eqref{conditionforDandnu}, we are reduced to search for $(D, \nu)$ satisfying \eqref{DD/nu}, \eqref{G1= G2ogamma} and such that 
\begin{equation}
\label{conditionsonDnu}
D, \, \nu\geq 1,\, D\nu >1, \, D\mid \nu \text{ and } 1\leq \nu \leq D^2.
\end{equation}
However, in the sequel of the proof, we will sometimes rather appeal to the property \eqref{minimalindexforgamma}, which is stronger than 
the inequality $1\leq \nu \leq D^2$ (see \S \ref{Section11.1} \& \S \ref{Section11.2}). 
 
\section{Interlude: equations with forms and points on varieties}
\subsection{Equations with binary forms}
We recall Lemma \ref{EverBomSch}. From this lemma, we deduce

\begin{lemma}
\label{879} 
Let $d \geq 3$ and $F \in {\rm Bin}(d, \Z)$. Let $n\geq 1$, $n_1$ and $n_2$ be fixed integers. Then, for every $\varepsilon > 0$ and for $Z$ tending to infinity, one has the asymptotic lower bound
$$
\vert \{m \in \Z: 1 \leq \vert m \vert \leq Z, m = F(x, y) \textup{ for some } x \equiv n_1 \bmod n, y \equiv n_2 \bmod n\} \vert \gg Z^{2/d - \varepsilon}.
$$
\end{lemma}

\begin{proof} 
We observe that there exists a constant $B_F$, depending on $F$ only, such that one has the implication
$$
\vert x \vert,\, \vert y \vert \leq B_F Z^{1/d} \Rightarrow \vert F(x, y) \vert \leq Z.
$$
The number of such pairs of integers $(x, y)$ satisfying, in addition, the imposed congruence conditions modulo $n$ is $\gg Z^{2/d}$. We then appeal to Lemma \ref{EverBomSch} to bound the number of integer solutions $(x, y)$ to the equation $F(x, y)= m$, where $m$ is fixed and satisfies $1\leq \vert m \vert \leq Z.$ Furthermore, the number of integral solutions to $F(x, y)=0$ in the box $\vert x\vert, \, \vert y\vert \leq B_F Z^{1/d}$ is negligible, since it is $\ll Z^{1/d}$. Combining these facts we complete the proof of Lemma \ref{879}.
\end{proof}

\subsection{Unique representation}
\begin{definition} 
Let $F \in {\rm Bin}(d, \Z)$ with $d \geq 3$. We say that an integer $h \ne 0$ is essentially represented by $F$ when the following two conditions are simultaneously satisfied 
\begin{itemize}
\item there exists $(x, y) \in \Z^2$ such that $F(x, y) =h$,
\item for all $(x, y, s, t)\in \Z^4$ such that $F(x, y) = F(s, t) =h$, there exists $\sigma \in {\rm Aut}(F, \Q)$ such that $(x, y) = \sigma (s, t)$.
\end{itemize}
\end{definition}

We now claim that almost integers represented by a form $F$ are essentially represented. Recall the notation \eqref{defNF} and introduce 
its variant
$$
N_{\rm uniq}(F, X) := \vert \{ h : \vert h \vert \leq X,\, h \text{ is essentially represented by } F \} \vert,
$$
to state

\begin{lemma} 
\label{notunique}
Let $d \geq 3$ and $F \in {\rm Bin}(d, \Q)$. There exists $\alpha < 2/d$ such that, for every $X \geq 1$ one has
$$
N(F,X)-N_{\rm uniq}(F,X) = O (X^\alpha).
$$
\end{lemma}

\begin{proof} 
We have
\begin{multline*}
N(F, X) - N_{\rm uniq}(F, X) \leq \\ 1 + \vert \{ (x, y) \in \Z^2 : 0 < \vert F( x, y) \vert \leq X \text{ and } F(x, y) \text { is not essentially represented}\} \vert.
\end{multline*}
Now use \cite[Lemma 2.4]{SX} to complete the proof. 
\end{proof}

\subsection{Lines on the surface associated to two binary forms}
\begin{definition}
Let $d \geq 3$ and let $F$ and $G$ two forms of ${\rm Bin}(d, \C)$. The surface associated to the pair $(F, G)$ is the surface $\mathbb W(F \cap G)\subset \mathbb P^3 (\C)$ defined by the equation
$$
\mathbb W (F \cap G) : F(X, Y) - G(Z, T) = 0,
$$
where $(X:Y:Z:T)$ are the usual coordinates of a point of $\mathbb P^3 (\C)$.
\end{definition}

It is easy to check that $\mathbb W (F\cap G)$ is a non-singular surface. In order to describe the lines on $\mathbb W (F \cap G)$ we consider the sets of zeroes of $F$ and $G$. By definition, the {\it set of zeroes} of a binary form $F$ is 
$$
\mathcal{Z}(F) := \{(x_1: x_2) \in \mathbb P^1 (\C): F(x_1, x_2) = 0\}.
$$
From the proof of \cite[Prop. 4.1]{BoSa}, we extract the following description of the lines on the surface $\mathbb W( F\cap G)$ (for the particular case where $F= G$, see \cite[Lemma 5.1, p.27]{CaHaMa}). To shorten notations, we write
$$ 
\mathcal{R}_\C := {\rm Isom}(G \rightarrow F, \C),
$$
and 
$$ 
\mathcal{R}_\Q := {\rm Isom}(G \rightarrow F, \Q).
$$
We state the following.

\begin{prop} 
\label{Bois-Sarti} 
Let $F$ and $G$ be two binary forms of ${\rm Bin}(d, \C)$. A line $\mathbb{D}$ of $\mathbb P^3 (\C)$ belongs to the surface $\mathbb{W}(F \cap G)$ if and only if it satisfies one of the following properties
\begin{enumerate}
\item \label{974} There exists $(x_1 : x_2) \in \mathcal{Z}(F)$ and $(x_3 : x_4) \in \mathcal{Z}(G)$ such that the line $\mathbb D$ goes through the point $(x_1 : x_2 : 0 : 0) $ and $(0 : 0 : x_3 : x_4)$.
\item \label{976} There exists $\rho \in \mathcal{R}_\C$ such that $\mathbb{D}$ has the following parametric equation
$$ 
\mathbb D_\rho : (u, v) \in \C^2 \mapsto (u, v, \rho(u, v)).
$$
\end{enumerate} 
\end{prop}

Observe that if the point $(z_1: z_2 : z_3 : z_4)$ belongs to a line of the type \ref{974}, we have $F(z_1, z_2) = G(z_3, z_4) =0$.

The following proposition studies common values of two forms coming from points which are not on these $\mathbb D_\rho$. 

\begin{prop}
\label{984} 
Let $d \geq 3$ and let $F$ and $G$ be two binary forms belonging to ${\rm Bin}(d, \Q)$. Let $\mathbb D_\rho$ ($\rho \in \mathcal{R}_\C$) be a line on the surface $\mathbb W (F \cap G)$, as defined in part \ref{976} of Proposition \ref{Bois-Sarti}. Then there exists a constant $\alpha < 2/d$ such that for all $B \geq 1$, one has 
\begin{multline} 
\label{1000}
\vert \{m \in \Z : \vert m \vert \leq B, \textup{ there exists } (x, y, z, t) \in \Z^4 \textup{ such that } \\
(x, y, z, t ) \notin \cup_{\rho \in \mathcal{R}_\C} \mathbb D_\rho, m= F(x, y) = G(z, t)\} \vert = O(B^\alpha).
\end{multline}
\end{prop}

Of course, this proposition is trivial when $F$ and $G$ are not ${\rm GL}(2, \Q)$--equivalent by Lemma~\ref{intersectionofimages}.

\subsubsection{Proof of Proposition \ref{984}} 
An important ingredient of the proof is the following result of Salberger that we state without specifying the value of $\xi_d$.

\begin{lemma}[{{\cite[Theorem 0.1]{Sal2}}}]
\label{Salberger}
Let $\mathbb X \subset \mathbb P^n$ be a geometrically integral projective surface over $\Q$ with degree $d \geq 3$. For $B \geq 1$, let
${\rm M}(\mathbb X, B)$ be the cardinality of the set of points of $\mathbb X$, with height at most $B$, which do not belong to the union of lines of $\mathbb X$. Then for some $\xi_d < 2$, only depending on $d$, one has the bound
$$
{\rm M}(\mathbb X, B) = O(B^{\xi_d}).
$$
\end{lemma} 

The first result of that type, with an exponent $\xi_d <2$, is due to Heath-Brown \cite{HB}. The value of $\xi_d$ was improved and the hypotheses were enlarged in \cite{Sal1} and \cite{Sal2}. The beginning of the proof of Proposition \ref{984} can be compared with the proof of \cite[Theorem 1.1, \S 2]{FW}. Let $\Delta >1$ be a parameter to be fixed later and let

\begin{equation}
\label{defB1}
B_1 = B^{1/d} \Delta.
\end{equation}

Let $\mathcal{P}(B)$ be the set of integers $m$ that appear at the left--hand side of \eqref{1000}. We decompose $\mathcal{P}(B)$ into two disjoint subsets
\begin{equation}
\label{1010} 
\mathcal{P}(B) = \mathcal P^<(B) \cup \mathcal P^{\geq}(B),
\end{equation}
where 
\begin{itemize}
\item $\mathcal P^<(B)$ contains all the $m\in \mathcal P(B)$ such that all the solutions to $m = F(x, y) = G(z, t)$ with $(x, y, z, t) \not\in \cup_{\rho \in \mathcal{R}_\C} \mathbb D_\rho$ satisfy $\max(\vert x\vert, \vert y \vert, \vert z \vert, \vert t \vert) < B_1$,
\item $\mathcal P^{\geq}(B)$ contains all the $m\in \mathcal P(B)$ such that the equations $m = F(x, y) = G(z, t)$ with $(x, y, z, t) \not\in \cup_{\rho \in \mathcal{R}_\C} \mathbb D_\rho$ admit at least one solution with $\max(\vert x\vert, \vert y \vert, \vert z \vert, \vert t\vert) \geq B_1$. 
\end{itemize}
After checking that the surface $\mathbb W (F \cap G)$ satisfies the hypotheses of Lemma \ref{Salberger} (see \cite [p.262]{FW0}), we deduce the bound
\begin{equation}
\label{1020}
\vert \mathcal P^<(B)\vert = O(B_1 ^{\xi_d}).
\end{equation}
To deal with the cardinality of $\mathcal P^{\geq }(B)$, we write the inequality
\begin{multline}
\label{1024}
\vert \mathcal P^{\geq }(B) \vert \leq \vert \{(x, y) \in \Z^2: 0 < \vert F(x, y) \vert \leq B, \max(\vert x\vert, \vert y \vert) \geq B_1\} \vert \\
+ \vert \{(z, t) \in \Z^2: 0 < \vert G(z, t) \vert \leq B , \max(\vert z \vert, \vert t \vert) \geq B_1\} \vert.
\end{multline}
The following lemma, which is based on a refinement of Liouville's inequality, asserts that almost all solutions $(x, y) \in \Z^2$ to the inequality $0 < \vert F(x, y)\vert \leq A$ (where $F$ belongs to ${\rm Bin}(d, \Z)$) are essentially located in a square slightly larger than the square $\{ (x, y) : \vert x \vert, \ \vert y \vert \leq A^{1/d}\}$. To be more precise, we have

\begin{lemma}[{{\cite[Prop. 2.6]{FW}}}]
Let $d \geq 3$ and let $F \in {\rm Bin}(d, \Z)$. There exist two computable constants $c_1(F)$ and $c_2(F)$, such that for every $\Delta > c_1(F)$ and for every $A \geq 1$, the following holds true
$$ 
\vert \{(x, y) \in \Z^2 : 0 < \vert F(x, y) \vert \leq A, \vert y \vert \geq A^{1/d} \Delta\} \vert 
\leq c_2(F) \left(A^{2/d} \Delta^{2 - d} +A ^{1/(d - 1)}\right).
$$
\end{lemma}

By the inequality \eqref{1024} and by this lemma we deduce 
\begin{equation}
\label{1037}
\vert \mathcal P^{\geq }(B ) \vert = O \left( B^{2/d} \Delta^{2-d} +B^{1/(d-1)}\right). 
\end{equation}
It remains to gather \eqref{defB1}, \eqref{1010}, \eqref{1020} and \eqref{1037}, to choose $\Delta$ of the shape $\Delta = B^{\delta_d}$, where $\delta_d > 0$ is sufficiently small, and to recall that $d \geq 3$ to complete the proof of Proposition \ref{984}. 

\subsubsection{Extension of Proposition \ref{984}} 
We are now interested with the images of points which do not lie on some line with rational coefficients. 

\begin{prop}
\label{984*}
Let $d \geq 3$ and let $F$ and $G$ be two binary forms belonging to ${\rm Bin}(d, \Q)$. Let $\mathbb D_\rho$ ($\rho \in \mathcal{R}_\Q$) be a line on the surface $\mathbb W (F \cap G)$, as defined in part \ref{976} of Proposition \ref{Bois-Sarti}. Then there exists a constant $\alpha <2/d$ such that
\begin{multline*} %\label{1000*}
\vert \{m \in \Z : \vert m \vert \leq B, \textup{ there exists } (x, y, z, t) \in \Z^4 \textup{ such that } \\
 (x, y, z, t ) \notin \cup_{\rho \in \mathcal{R}_\Q} \mathbb D_\rho, m= F(x, y) = G(z, t)\} \vert =O(B^\alpha).
\end{multline*}
\end{prop}

\begin{proof} 
We identify the isomorphim $\rho$ with a matrix of ${\rm GL}(2, \C)$. We will deduce Proposition \ref{984*} from Proposition \ref{984} as soon as we produce for each $\rho = \begin{pmatrix} a_1 & a_2 \\ a_3 & a_4\end{pmatrix}$ belonging to $\mathcal{R}_\C$, but not to $\mathcal{R}_\Q$, an admissible bound. By hypothesis, there exists at least one entry of the matrix $\rho$ which is not rational, say $a_1$ or $a_2$.
\begin{enumerate}
\item Suppose $\dim_\Q \{\Q + \Q a_1 + \Q a_2\} = 2$. This means that one of the $a_i$ (say $a_2$) can be written as
$$
a_2 = \alpha + \beta a_1,
$$
where $\alpha $ and $\beta$ are rational numbers. The line given by $\rho$ is $(x, y, a_1 x + a_2 y, a_3 x + a_4 y)$. Then the condition $a_1 x + a_2 y \in \Z$ implies $x + \beta y = 0$. The corresponding value $m$ is determined by
$$
m = F(x, y) = F(-\beta y, y) = y^d F(-\beta, 1),
$$
and the number of these $m$ with $\vert m \vert \leq B$ is in $O(B^{1/d}) = O(B^\alpha)$.
\item Suppose $\dim_\Q \{\Q + \Q a_1 + \Q a_2\} = 3$. The condition $a_1 x+ a_2 y\in \Z$ directly implies $x = y = 0$, which leads to the value $m = 0$.
\end{enumerate}
Having covered all cases, the proof is complete.
\end{proof}

\section{\texorpdfstring{From $(G_1, G_2)$ to covering by lattices}{Towards coverings of lattices}} 
We recall Definition \ref{associatedlattice} and Lemma \ref{959}. We consider a more general situation than \eqref{G1G2Dnu} and prove the following crucial proposition, where we exhibit explicit coverings by lattices.

\begin{prop}
\label{proposition6.1}
Let $d \geq 3$ and let $G_1$ and $G_2$ belong to ${\rm Bin}(d, \Z)$ such that $G_1(\Z^2) = G_2(\Z^2)$. Let $\gamma \in {\rm Isom}(G_2 \rightarrow G_1, \Q)$. We then have the equalities
\begin{equation}
\label{claim}
\Z^2 = \bigcup_{\sigma_1 \in {\rm Aut}(G_1, \Q)} L (\gamma \sigma_1) = \bigcup_{\sigma_2 \in {\rm Aut}(G_2, \Q)} L (\sigma_2 \gamma ) , 
\end{equation}
and 
\begin{equation}
\label{claim*}
\Z^2 = \bigcup_{\sigma_1 \in {\rm Aut}(G_1, \Q)} L ( \sigma_1\gamma^{-1}) = \bigcup_{\sigma_2 \in {\rm Aut}(G_2, \Q)} L (\gamma^{-1} \sigma_2 ) .
\end{equation}
\end{prop}

\begin{proof}
We will concentrate on the proof of the first equality of \eqref{claim} that we write as 
\begin{equation}
\label{firstequation}
\Z^2 = \mathcal L^\cup,
\end{equation}
with
$$
\mathcal L^\cup := \bigcup_{\sigma_1 \in {\rm Aut}(G_1, \Q)} L(\gamma \sigma_1).
$$
The other equalities presented in \eqref{claim} and \eqref{claim*} are consequences of Lemma \ref{959} and of the equality $G_2 = G_1 \circ \gamma^{-1}$.

For the sake of contradiction, suppose that the claim \eqref{firstequation} is false. Let $(c_1, c_2) \in \Z^2$, which does not belong to $\mathcal L^\cup$. The intersection of lattices $\mathcal L^\cap := \bigcap_{\sigma_1 \in {\rm Aut}(G_1, \Q)} L(\gamma \sigma_1)$ is also a lattice. Then, by Lemma \ref{existenceofC}, there exists $C > 1$ such that $C\Z^2 \subseteq \mathcal L^\cap$. Consider the following subset of $\Z^2$
\begin{equation*}
\mathcal E := \{(u, v) \in \Z^2 : u \equiv c_1 \bmod C, v \equiv c_2 \bmod C\}.
\end{equation*}
By construction $\mathcal E$ satisfies
\begin{equation}
\label{emptyintersection}
\mathcal E \cap \mathcal L^\cup = \emptyset.
\end{equation}
This construction of $\mathcal E$, which is a translate of a lattice, shows that many points (a positive proportion, in a natural sense) do not belong to $ \mathcal L^\cup$. The image of $\mathcal E$ by $G_1$ is also a rather dense subset, since by Lemma \ref{879}, we have the asymptotic lower bound
$$
\vert \{m \in \Z : \vert m \vert \leq Z, \, m = G_1(u, v) \text{ for some } (u, v) \in \mathcal E\} \vert \gg_\varepsilon Z^{2/d-\varepsilon}.
$$
Combining with the hypothesis $G_1(\Z^2) = G_2(\Z^2)$, we obtain 
\begin{align}
\label{Salb}
\vert \{m \in \Z : \vert m \vert \leq Z, m = G_1(u, v) = G_2(s, t) \text{ for some } (u, v, s, t) \in \mathcal E \times \Z^2\} \vert \gg_{\varepsilon} Z^{2/d - \varepsilon}.
\end{align}
To shorten notations, we set $\mathcal{R}_\Q := {\rm Isom}(G_2 \rightarrow G_1, \Q)$. We decompose into two (not necessarily disjoint) subsets
\begin{align*}
&\{m : \vert m \vert \leq Z, m = G_1(a, b) \text{ for some } (a, b) \in \Z^2\} \\
&= \{m : \vert m \vert \leq Z, m = G_1(a, b) = G_2(c, d) \text{ for some } (a, b, c, d) \in \Z^4\} \\
&= \{m : \vert m \vert \leq Z, m = G_1(a, b) = G_2(c, d) \text{ for some } (a, b, c, d) \in \left(\cup_{\rho \in \mathcal{R}_\Q} \mathbb D_\rho\right) \cap \Z^4\} \\
&\bigcup \{m : \vert m \vert \leq Z, m = G_1(a, b) = G_2(c, d) \text{ for some } (a, b, c, d) \not\in \left(\cup_{\rho \in \mathcal{R}_\Q} \mathbb D_\rho\right), (a, b, c, d) \in \Z^4\} \\
&:= \mathcal C_1 \bigcup \mathcal C_2, 
\end{align*}
by definition. Proposition \ref{984*} directly gives for some $\alpha < 2/d$
$$
\vert \mathcal C_2 \vert \ll Z ^\alpha.
$$
But this upper bound is less than the lower bound produced by \eqref{Salb}, thus we deduce the following asymptotic lower bound
\begin{multline*} 
\vert \{m : \vert m \vert \leq Z, \text{ there exists } (a, b, c, d) \in ( \cup_{\rho \in \mathcal{R}_\Q} \mathbb D_\rho) \cap \Z^4 \text { and } (u, v, s, t) \in \mathcal E \times \Z^2 \\
\text{ such that } m = G_1(a, b) = G_2(c, d) = G_1(u, v) = G_2(s, t)\} \vert \gg_\varepsilon Z^{2/d - \varepsilon}.
\end{multline*}
We appeal to Lemma \ref{notunique} (applied to the form $F = G_1$) to incorporate the constraint for $m$ to be essentially represented and we obtain the asymptotic lower bound 
\begin{multline*} 
\vert \{m : \vert m \vert \leq Z, m \text{ is essentially represented by } G_1 \\
\text{ and there exists } (a, b, c, d) \in ( \cup_{\rho \in \mathcal{R}_\Q} \mathbb D_\rho) \cap \Z^4 \text { and } (u, v,s, t) \in \mathcal E \times \Z^2 \\
\text{ such that } m = G_1(a, b) = G_2(c, d) = G_1(u, v) = G_2(s, t)\} \vert
\gg_\varepsilon Z^{2/d-\varepsilon}.
\end{multline*} 
For $m \neq 0$, let us consider the equation $m = G_1(a, b) = G_1(u, v)$. The above lower bound implies the existence of some $(u, v) \in \mathcal E$, some $\rho \in \mathcal{R}_\Q$, some $(a, b, c, d) \in \mathbb D_\rho \cap \Z^4$ and some $\tau_1 \in {\rm Aut}(G_1, \Q)$ such that
\begin{equation}
\label{tau(u, v)}
\tau_1(u, v) = (a, b).
\end{equation}
Since $(a, b, c, d) $ belongs to $\mathbb D_\rho \cap \Z^4$, we have $(a, b) = \rho^{-1}(c, d).$ Returning to \eqref{tau(u, v)}, we deduce the equality
$$
\rho \tau_1(u, v) = (c, d) \in \Z^2.
$$
This shows that $(u, v)$ belongs to $L (\rho \tau_1)$, where $\tau_1$ belongs to ${\rm Aut}(G_1, \Q)$. This implies that $(u, v)$ belongs to $\mathcal L^{\cup}$, which contradicts the hypothesis $(u, v) \in \mathcal E$, as written in \eqref{emptyintersection}. Thus the proof of \eqref{firstequation} is complete.
\end{proof}

We arrive at the last stage of the proof of Theorem \ref{source}, but our discussion requires to split the condition $(C1)$ in two distinct subcases according to the divisibility by $3$ or not of the group of automorphisms.

\section{\texorpdfstring{Small automorphism groups $\bf C_1, \bf C_2, \bf C_4, \bf D_1, \bf D_2$}{Small automorphism groups}} 
\label{smallAut}
We prove

\begin{theorem} 
Let $d \geq 3$. Then ${\rm Bin}(d, \Z)$ contains no extraordinary form $F$ such that
\begin{equation}
\label{5groups}
{\rm Aut}(F, \Q) \simeq_{{\rm GL}(2, \Q)} \bf C_1, \, \bf C_2, \, \bf C_4, \, \bf D_1,\, \bf D_2.
\end{equation}
\end{theorem}

\begin{proof} 
For the sake of contradiction, suppose that such an extraordinary form $F$ exists. To follow notations of \S\ref{FromF1F2toG1G2}, we put $F_1 = F$ and we suppose that there exists $F_2$ satisfying \eqref{hypoforF1F2}. Then we construct two forms $G_1$ and $G_2$, two integers $D$ and $\nu$ satisfying \eqref{music} and \eqref{G1G2Dnu}. Now we apply Proposition \ref{proposition6.1} which leads to the two coverings 
$$
\Z^2 =\bigcup_{\sigma_1 \in {\rm Aut}(G_1, \Q)} L(\gamma \sigma_1) \quad \text{ and } \quad \Z^2 = \bigcup_{\sigma_1 \in {\rm Aut}(G_1, \Q)} L (\sigma_1 \gamma^{-1}), 
$$ 
where $\gamma$ is defined in \eqref{DD/nu}. By construction, we know that ${\rm Aut}(G_1, \Q)$ satisfies a conjugacy property similar to \eqref{5groups} and by Lemma \ref{gamma-gamma}, the above covering leads to a covering of $\Z^2$ by at most two lattices (indeed $-{\rm id}$ belongs to ${\bf C}_2$, ${\bf C}_4$, ${\bf D}_2$ and to the ${\rm GL}(2, \Q)$--conjugates of these groups). By Lemma \ref{2lattices}, we deduce that one of these lattices is trivial, so there exist $\sigma_1, \sigma'_1 \in {\rm Aut}(G_1, \Q)$, such that $L(\gamma \sigma_1) = L(\sigma'_1 \gamma^{-1}) = \Z^2$. We now apply Lemma \ref{lemma5.4} to deduce that $\gamma \sigma_1$ and $ \sigma'_1 \gamma^{-1}$ have integer coefficients, in particular their determinants are integers different from zero.

From the equalities $G_1 = G_2 \circ \gamma \sigma_1$ and $G_2 = G_1 \circ \sigma'_1 \gamma^{-1}$ applied to \eqref{invarianceofarea}, we deduce that the areas of the fundamental domains satisfy $A_{G_1} = A_{G_2}$ and $\vert \det(\gamma \sigma_1) \vert = \vert \det(\sigma'_1 \gamma^{-1}) \vert = 1$. We finally obtain that $\gamma \sigma_1$ and $\sigma'_1 \gamma^{-1}$ belong to ${\rm GL}(2, \Z)$ and $G_1$ and $G_2$ are ${\rm GL}(2, \Z)$--equivalent, which contradicts the hypothesis.
\end{proof}

\section{\texorpdfstring{The case of the groups ${\bf C}_3$ and ${\bf C}_6$}{The case of C3 and C6}} 
We now work in the complementary case of \eqref{5groups} in $(C1)$. In other words, we will prove Theorem \ref{source} in the situation where ${\rm Aut}(F_1, \Q) \simeq_{\rm{GL}(2, \Q)} {\bf C}_3$ or ${\bf C}_6$. But $-{\rm id}$ belongs to ${\bf C}_6$. So by Lemma \ref{gamma-gamma}, we have only to consider the ${\bf C}_3$ case when studying the coverings given by Proposition \ref{proposition6.1}. Recall that $\gamma$ defined in \eqref{DD/nu} is such that $L (\gamma)$ has minimal index among the lattices appearing in the right--hand sides of \eqref{claim} and \eqref{claim*}. This is a consequence of \eqref{minimalindexforgamma}. Recall also the relations \eqref{music}, \eqref{G1G2Dnu}, \eqref{G1= G2ogamma} and \eqref{conditionsonDnu}. Our purpose is to prove

\begin{equation}
\label{Asuccess}
D = 2, \nu = 2 \text{ and } {\rm Aut}(G_2, \Q) \subset {\rm GL}(2, \Z),
\end{equation}
under the above hypotheses. We recognize the second case of Theorem \ref{source}. We separate our proof according to the value of the index $[\Z^2 : L(\gamma)]$.

\subsection{\texorpdfstring{Case 1: $[\Z^2 : L (\gamma)] =1$}{Case 1}}
\label{Section11.1}
Write $\tau$ for a non-trivial element of ${\rm Aut}(G_2, \Q)$. So we have
\begin{equation}
\label{idtautau2}
{\rm Aut}(G_2, \Q) = \{{\rm id}, \tau, \tau^2\}.
\end{equation}
We deduce that $L(\gamma^{-1})$, $L( \gamma^{-1} \tau) $ and $L( \gamma^{-1} \tau^2)$ are proper sublattices of $\Z^2$, since otherwise $G_1$ and $G_2$ would be ${\rm GL}(2, \Z)$--equivalent, by the argument made at the end of \S \ref{smallAut}. Define the lattices $\Lambda_0$, $\Lambda_1$ and $\Lambda_2$ as in Lemma \ref{3lattices}. Since $L(\gamma^{-1})$, $L( \gamma^{-1} \tau) $ and $L( \gamma^{-1} \tau^2)$ are three proper lattices of $\Z^2$ and since \eqref{claim*} is a covering by three lattices, we deduce, from Lemma \ref{3lattices} the equality between two sets of three lattices
\begin{equation}
\label{3lattices=3lattices}
\{L(\gamma^{-1}), L (\gamma^{-1} \tau), L (\gamma^{-1} \tau^2)\} = \{\Lambda_0, \Lambda_1, \Lambda_2\}.
\end{equation}
By the Definition \ref{associatedlattice} and formula \eqref{DD/nu}, we have the equality
$$
L(\gamma^{-1}) = \{(x, y) : x\equiv 0 \bmod D, \, \nu y \equiv 0 \bmod D \},
$$
and by \eqref{3lattices=3lattices}, the only possibility is $L(\gamma^{-1}) = \Lambda_0$, so $D =2$ and $2 \mid \nu$. Since $L(\gamma) = \Z^2$, we get the unique possibility $(D, \nu) = (2, 2)$, and \eqref{3lattices=3lattices} becomes
\begin{equation}
\label{2lattices=2lattices}
\{L(\gamma^{-1} \tau), L(\gamma^{-1} \tau^2)\} = \{\Lambda_1, \Lambda_2\}.
\end{equation}
Returning to \eqref{G1G2Dnu}, we have
$$
G_1(X, 2Y) = G_2(2X, 2Y),
$$
that we write as 
$$
G_1 = G_2 \circ \gamma,
$$
with 
$$
\gamma = \begin{pmatrix} 2 & 0 \\ 0 & 1\end{pmatrix}.
$$
Write $\tau$, appearing in \eqref{idtautau2}, as
\begin{equation}
\label{taucanonical}
\tau = \frac 1a \begin{pmatrix} a_1& a_2 \\ a_3 & a_4\end{pmatrix},
\end{equation}
in the canonical form (see Lemma \ref{canonicalmatrix}), since $\det \tau = 1 = (a_1 a_4- a_2 a_3)/a^2$. We compute explicitely the two matrices
$$
\gamma^{-1} \tau = \frac{1}{a} \begin{pmatrix} a_1/2 & a_2/2 \\ a_3 & a_4 \end{pmatrix}
\text{ and } 
\gamma^{-1} \tau^2 = \frac{1}{a} \begin{pmatrix} a_4/2 & -a_2/2 \\ -a_3 & a_1\end{pmatrix},
$$
that we write in canonical form as
\begin{equation}
\label{twocanonicalforms}
\gamma^{-1} \tau =\frac 1{2a/s} \begin{pmatrix} a_1/s & a_2/s \\ 2a_3/s &2 a_4 /s\end{pmatrix}
\text{ and } 
\gamma^{-1} \tau^2 =\frac 1{2a/s'} \begin{pmatrix} a_4/s' & -a_2/s' \\ -2a_3/s' & 2a_1/s'\end{pmatrix},
\end{equation}
where $s := {\rm gcd}(a_1, a_2, 2a_3, 2a_4)$ and $s' := {\rm gcd}(2a_1, a_2, 2a_3, a_4)$. Since we also have the equality ${\rm gcd}(a_1, a_2, a_3, a_4) = 1$, we deduce that the integers $s$ and $s'$ satisfy the property
\begin{equation}
\label{ss=12}
s, s' \in \{1, 2\}.
\end{equation}
Thanks to the canonical forms \eqref{twocanonicalforms}, we deduce from Lemma \ref{lemma5.10} the equalities
$$
\vert \det L(\gamma^{-1} \tau) \vert = 2 a/s \quad \text{ and } \quad \vert \det L(\gamma^{-1} \tau^2) \vert = 2a/s'.
$$
However, on the right--hand side of \eqref{2lattices=2lattices}, the two lattices have a determinant equal to $2$. So we have the equalities
$$
2a = 2s = 2s'.
$$
By the possible values of $s$ and $s'$ given in \eqref{ss=12}, we are led to split our discussion into two cases.
\vskip .3cm
\noindent $\bullet$ If $s = s' = a = 2$. By the definition of $s$ and $s'$ and by the equality $\gcd(a_1, a_2, a_3, a_4) =1$ we see that this case happens when
$$
\begin{cases}
a_1 \equiv a_2\equiv a_4 \equiv 0 \bmod 2, \\
a_3 \equiv 1 \bmod 2.
\end{cases}
$$
With these conditions, we see that the matrices $\gamma^{-1} \tau$ and $\gamma^{-1} \tau^2$ (see \eqref{twocanonicalforms}) have the shape
$$
\gamma^{-1} \tau = 
\begin{pmatrix}
* & * \\
m + \frac{1}{2} & n
\end{pmatrix}
\text{ and }
\gamma^{-1} \tau^2 = 
\begin{pmatrix}
* & * \\
-m - \frac{1}{2} & n'
\end{pmatrix},
$$
where $m$, $n$ and $n'$ are integers and where $*$ are unspecified numbers. We check that the point $(1, 1)$ does not belong to $L(\gamma^{-1} \tau) \cup L(\gamma^{-1} \tau^2)$ but it belongs to $\Lambda_1 \cup \Lambda_2$. This contradicts \eqref{2lattices=2lattices}. So this case never happens.
\vskip .3cm
\noindent $\bullet$ If $s = s' = a = 1$. Returning to \eqref{taucanonical}, we see that $\tau$ belongs to ${\rm GL}(2, \Z)$ and by \eqref{idtautau2}, ${\rm Aut}(G_2, \Q)$ is a subgroup of $ {\rm GL}(2, \Z)$. This proves \eqref{Asuccess} in that case.

\subsection{\texorpdfstring{Case 2: $[\Z^2 : L (\gamma)] \geq 2$}{Case 2}} 
\label{Section11.2}
Recall that $\gamma$ is defined in \eqref{DD/nu} and that ${\rm Aut}(G_2, \Q)$ is defined by \eqref{idtautau2}. Similarly, we write
$$
{\rm Aut}(G_1, \Q) = \{{\rm id}, \sigma, \sigma^2\}.
$$
So we have 
$$
{\rm Isom}(G_2\rightarrow G_1 , \Q) = \{\gamma, \gamma \sigma, \gamma\sigma^2\} \quad \text{ and } \quad {\rm Isom}(G_1\rightarrow G_2, \Q) = \{\gamma^{-1}, \gamma^{-1} \tau, \gamma^{-1} \tau^2\}.
$$
All the lattices $L(\gamma)$, $L(\gamma \sigma)$, $L(\gamma\sigma^2)$, $L(\gamma^{-1})$, $L(\gamma^{-1} \tau)$ and $L(\gamma^{-1} \tau^2)$ are proper sublattices, since by construction all their index are $\geq [\Z^2 : L(\gamma)] \geq 2$ (see \eqref{minimalindexforgamma}). A combination of Proposition \ref{proposition6.1} and Lemma \ref{3lattices} gives the two equalities between sets of three lattices
$$ 
\{L(\gamma), L(\gamma \sigma), L(\gamma\sigma^2)\} = \{\Lambda_0, \Lambda_1, \Lambda_2\}, 
$$
and 
\begin{equation}
\label{1386} 
\{L(\gamma^{-1}), L(\gamma^{-1} \tau), L(\gamma^{-1}\tau^2)\} = \{\Lambda_0, \Lambda_1, \Lambda_2\}.
\end{equation}
These equalities imply that
$$
L(\gamma), L(\gamma^{-1}) \in \{\Lambda_0, \Lambda_1, \Lambda_2\}.
$$
The definition \eqref{DD/nu} of $\gamma$ implies that $L(\gamma)$ and $L(\gamma^{-1})$ are distinct from $\Lambda_2$, so we have
$$
L(\gamma), L(\gamma^{-1}) \in \{\Lambda_0, \Lambda_1\}.
$$
We again use the explicit definition of $L (\gamma)$ and $L (\gamma^{-1})$ and the conditions \eqref{conditionsonDnu} to deduce that we necessarily have
$$
(D, \nu) = (2, 4).
$$
We now show that this configuration is impossible as follows. With the above value of $(D, \nu)$ we have
$$
\gamma = \begin{pmatrix} 2 & 0 \\ 0 & 1/2 \end{pmatrix}.
$$
Write $\tau$ in the canonical form
$$
\tau = \frac{1}{a} \begin{pmatrix} a_1 & a_2 \\ a_3 & a_4 \end{pmatrix} \text{ with } a \geq 1, a^2 = \pm (a_1a_4 -a_2a_3) \text{ and }
{\rm gcd}(a_1,a_2,a_3,a_4) =1, 
$$
and we compute
\begin{equation}
\label{explicit1}
\gamma^{-1} \tau = \frac{1}{a} \begin{pmatrix} a_1/2 & a_2/2 \\ 2a_3 & 2a_4 \end{pmatrix} \text{ and }\gamma^{-1} \tau^2 = \gamma^{-1} \tau^{-1}=
\frac{1}{a} \begin{pmatrix} a_4/2 & -a_2/2 \\ -2a_3& 2 a_1 \end{pmatrix}.
\end{equation}
We write these two matrices in their canonical form
$$
\gamma^{-1} \tau = \frac 1{2a/t} \begin{pmatrix} a_1/t & a_2/t \\ 4a_3/t & 4a_4/t\end{pmatrix} \text{ with } t = {\rm gcd}(a_1, a_2, 4a_3, 4a_4),
$$
and
$$
\gamma^{-1} \tau^2 = \gamma^{-1} \tau^{-1} = \frac{1}{2a/t'} 
\begin{pmatrix} 
a_4/t' & -a_2/t' \\ 
-4a_3/t' & 4a_1/t'
\end{pmatrix} 
\text{ with } t ' = {\rm gcd}(4a_1, a_2, 4a_3, a_4).
$$
We necessarily have $t, t' \in \{1, 2, 4\}$. Using \eqref{1386} and Lemma \ref{lemma5.10} to compute the index, we obtain the equalities
$$ 
\frac{(2a/t)^2}{{\rm gcd}(2a/t, 4a^2/t^2)} = \frac{(2a/t')^2}{{\rm gcd}(2a/t', 4a^2/t'^2)} = 2
$$ 
which reduce to 
$$
2a/t = 2a/t' = 2,
$$
which gives finally 
$$
a = t = t' \in \{1, 2, 4\}.
$$
We investigate three cases
\vskip .2cm
\noindent $\bullet$ {\it The case $a = 4$ is excluded}. Suppose that $a = 4$. From the equalities 
$$
4 = \gcd(a_1, a_2, 4 a_3, 4a_4) = \gcd(4a_1, a_2, 4a_3, a_4), 
$$
we deduce that $a_1$, $a_2$ and $a_4$ are even, and from the coprimality of the $a_i$, we know that $a_3$ is odd. We now exploit the explicit expressions of the lattices $L (\gamma^{-1})$, $L(\gamma^{-1} \tau) $ and $L(\gamma^{-1} \tau^2)$ (see \eqref{explicit1}). We check that the point $(1, 0)$ belongs to none of these three lattices since $2a_3/4$ does not belong to $\Z$. This contradicts \eqref{1386}.
\vskip .2cm
\noindent $\bullet$ {\it The case $a = 2$ is excluded}. Suppose that $a = 2$. As above, we know that $a_1$, $a_2$ and $a_4$ are even, but $a_3$ is odd. The point $(0,1)$ belongs to $L(\gamma^{-1})$ but it does not belong to $\Lambda_1 \cup \Lambda_2$. By the equality \eqref{1386}, we deduce that $(0, 1)$ does not belong to $L(\gamma^{-1} \tau) \cup L(\gamma^{-1} \tau^2)$. The explicit expression of $\gamma^{-1} \tau$ and $\gamma^{-1} \tau^2$ leads to the conclusion that we necessarily have $a_2\equiv 2 \bmod 4$. However, we have the equality
$$
\vert a_1a_4 - a_2a_3 \vert = a^2 =4.
$$
Therefore we arrive at a contradiction since we proved the congruences $a_1 \equiv a_4 \equiv 0 \bmod 2$, $a_3 \equiv 1 \bmod 2$ and $a_2 \equiv 2 \bmod 4$.
\vskip .2cm
\noindent $\bullet$ {\it The case $a = 1$ is excluded}. Suppose that $a = 1$. Then the matrix $\tau$ belongs to ${\rm GL}(2, \Z)$. We apply Lemma \ref{433} with the choice $F = G_2$ to obtain the equivalence of the two forms
\begin{equation}
\label{2214}
G_2(2X, Y) \sim_{{\rm GL}(2, \Z)} G_2(X, 2Y).
\end{equation}
The equality \eqref{G1G2Dnu}, with $D = 2$ and $\nu = 4$, gives the equality
$$
G_1(X, 4Y) = G_2(2X, 2Y),
$$ 
that we write as 
$$
G_1(X, 2Y) = G_2(2X, Y),
$$
which combined with \eqref{2214} leads to
$$
G_1(X, 2Y) \sim_{{\rm GL}(2, \Z)} G_2(X, 2Y),
$$
which leads to $G_1 \sim_{{\rm GL}(2, \Z)} G_2$, which is contrary to \eqref{G1G2Dnu}. 
\par 
Thus we investigated the three possible values of $a$ and we conclude that in Case 2, the equalities \eqref{G1G2Dnu} have no solution. The proof of \eqref{Asuccess} is complete, and hence also the proof of Theorem~\ref{source}.

\section{Coprime value sets}
\label{sCoprime}
Throughout this section we will assume the truth of Theorem \ref{source}, Theorem \ref{A} and Theorem \ref{B}. We recall some notation. We defined $W(F) := \{F(x, y) : x, y \in \Z, \gcd(x, y) = 1\}$ and we write $\sim_{\text{co}}$ for the resulting equivalence relation, i.e. $F \sim_{\text{co}} G$ if $W(F) = W(G)$. We also recall the inclusions (\ref{eCoprimeValue})
\[
[F]_{\GL} \subseteq [F]_{\text{co}} \subseteq [F]_{\text{val}}.
\]
In order to prove Theorem \ref{tCoprimeValue}, it suffices to show that

\begin{theorem}
Let $d \geq 3$ and let $F, G \in {\rm Bin}(d, \Q)$. Suppose that $F \sim_{\textup{co}} G$. Then we have $F \sim_{\GL} G$.
\end{theorem}

\begin{proof}
Suppose that $F \not \sim_{\GL} G$. Since we certainly have that $F \sim_{\text{val}} G$, it follows from Theorem \ref{source}, Theorem \ref{A} and Theorem \ref{B} that, upon interchanging $F$ and $G$ and replacing them by $\GL$--equivalent forms if necessary, we have
$$
G(X, Y) = F(X, 2Y).
$$
Furthermore, ${\rm Aut}(F)$ has an element of order $3$ contained in $\GL$. We now fix any tuple $(x_0, y_0)$ with $x_0$ even, $y_0$ odd, $\gcd(x_0, y_0) = 1$ and $G(x_0, y_0) \neq 0$. By our assumption $W(F) = W(G)$, we may find some $(x_1, y_1)$ with $\gcd(x_1, y_1) = 1$ such that
\[
F(x_1, y_1) = G(x_0, y_0).
\]
Since ${\rm Aut}(F)$ has an element of order $3$ contained in $\GL$, we may use the argument from Lemma \ref{433} to ensure that $y_1$ is odd (by changing, if necessary, the pair $(x_1, y_1)$ to $\sigma(x_1, y_1)$ for a suitable $\sigma \in {\rm Aut}(F) \cap \GL$ of order $3$). We then consider
\[
G(2x_1, y_1) = F(2x_1, 2y_1) = 2^d F(x_1, y_1) = 2^d G(x_0, y_0),
\]
which is primitively represented by $G$. But then it must also be primitively represented by $F$, so we find some $(x_2, y_2)$ with $\gcd(x_2, y_2) = 1$ and $y_2$ odd such that
\[
F(x_2, y_2) = G(2x_1, y_1).
\]
We look at
\[
G(2x_2, y_2) = F(2x_2, 2y_2) = 2^d F(x_2, y_2) = 2^d G(2x_1, y_1) = 2^{2d} G(x_0, y_0).
\]
Continuing in this way, we obtain a sequence of pairs $(x_N, y_N)_{N \geq 1}$ such that
\[
F(x_N, y_N) = 2^{(N - 1) d} G(x_0, y_0), \quad \gcd(x_N, y_N) = 1.
\]
This gives infinitely many solutions to the Thue--Mahler equation \cite{LM}, contradiction.
% Use Theorem 8.10 from https://pub.math.leidenuniv.nl/~evertsejh/dio19-8.pdf. This argument relies on the subspace theorem and does not assume irreducibility like many other sources
% Alternative argument for C3:
%Furthermore, Theorem \ref{source} and the assumption $(C1)$ imply that ${\rm Aut}(F)$ is a subset of $\GL$. A straightforward adaptation of Lemma \ref{879} shows that
%\[
%S := \{m \in \Z : 1 \leq m \leq Z, m = G(x, y) \text{ for some } x \equiv 0 \bmod 2, y \equiv 1 \bmod 2, \gcd(x, y) = 1\}
%\]
%satisfies $|S| \gg Z^{2/d - \epsilon}$. We claim that every non-zero $m \in S$ is represented by $F$, but not essentially represented by $F$. Since $|S| \gg Z^{2/d - \epsilon}$, our claim contradicts Lemma \ref{notunique}. Therefore it suffices to prove the claim.
%
%So take $m \in S$ with $m \neq 0$. By assumption we have $m = G(x, y)$ for some $x \equiv 0 \bmod 2$, $y \equiv 1 \bmod 2$ and $\gcd(x, y) = 1$. In particular, we have that $m \in W(G)$. Since $W(F) = W(G)$ by assumption, we find some $x', y' \in \Z$ with $\gcd(x', y') = 1$ such that $F(x', y') = m$. However we also have
%\[
%m = G(x, y) = F(x, 2y).
%\]
%If $m$ were to be essentially represented, there would be an automorphism $\sigma$ of $F$ such that $(x, 2y) = \sigma(x', y')$. But since ${\rm Aut}(F)$ is a subset of $\GL$, it can not send the primitive vector $(x', y')$ to $(x, 2y)$ (recall that $x$ is even). This is the desired contradiction.
\end{proof}
%%%%%%%%%%%%%%%%%%%%

\section{Value sets with multiplicities} 
\label{sMulti}
In this section we prove Theorem \ref{tMultiValue}, which is equivalent to the following

\begin{theorem}
\label{1867}
Let $d \geq 3$ and let $F, G \in {\rm Bin}(d, \Z)$. Suppose that $F \sim_{\textup{mu}} G$. Then we have $F \sim_{\GL} G$.
\end{theorem}

\begin{proof} 
The beginning of the proof mimics what is contained in \S \ref{sCoprime}.
Suppose that $F \not \sim_{\GL} G$. Since we certainly have that $F \sim_{\text{val}} G$, it follows from Theorem \ref{source}, Theorem \ref{A} and Theorem \ref{B} that, upon interchanging $F$ and $G$ and replacing them by $\GL$--equivalent forms if necessary, we have
$$
G(X, Y) = F(X, 2Y).
$$
Therefore we have $R(F; m) = R (G; m)$ by hypothesis. 

Take integers $x_0$ and $y_0$ with $y_0$ odd and with $m := F(x_0, y_0) \neq 0$. We claim that $m$ is represented more frequently by $F$ than by $G$. Indeed, we have an injective map $\varphi$
\[ \varphi :
\{(x, y): G(x, y) = m\} \rightarrow \{(x, y) : F(x, y) = m\}
\]
given by sending $(x, y)$ to $(x, 2y)$. Since $(x_0, y_0)$ is not in the image of $\varphi$, we have proven the claim, and thus the theorem.
%By \cite[Lemma 2.1]{SX} (which gathers the works of Mahler \cite{Ma} and of Thunder \cite{Th}), we have the asymptotics ($X\rightarrow \infty$)
%\begin{equation}
%\label{Mahler}
%\sum_{0 < \vert m \vert \leq X} R (F; m )\left(=\vert \{\, (x,y)\in \Z^2 : 1\leq \vert F (x,y)\vert \leq X\, \}\vert \right) \sim A_F X^{2/d},
%\end{equation}
%where 
%$A_F$ is the area of the fundamental domain
%$$
%A_F = \iint_{\vert F(x,y) \vert \leq 1} dx dy.
%$$
%Applying \eqref{Mahler} to $F$ and $F^\star$ we arrive at a contradiction since $A_F$ and $A_{F^\star}$ differ by a factor $2$.
\end{proof}

The equivalence relation $\sim_{\rm mu}$ is very restrictive when compared with $\sim_{\rm val}$ and one may wonder if there exists a more direct proof of Theorem \ref{1867}.

\end{document}